\documentclass[11pt]{article}
\usepackage[english]{babel}

\usepackage{amsmath} 
\usepackage{tabularx}
\usepackage{amssymb}
\usepackage{amsfonts}
\usepackage{relsize}
\usepackage{array}
\usepackage{amsthm}
\usepackage[table]{xcolor}
\usepackage{multicol}
\usepackage{csvsimple}
\usepackage{array}
\usepackage{xcolor}
\usepackage{tikz}
\usepackage{diagbox}
\usepackage{enumerate}
\newcommand\crossmark[1][]{%
	\tikz[scale=0.4,#1]{
		\fill(0,0)--(0.1,0) .. controls (0.5,0.4) .. (1,0.7)--(0.9,0.7) ..  controls (0.5,0.5) ..(0,0.1) --cycle;
		\fill(1,0.1)--(0.9,0.1) .. controls (0.5,0.3) .. (0,0.7)--(0.1,0.7) .. controls (0.5,0.4) ..(1,0.2) --cycle;
	}%
}
\usepackage{array}
\newcolumntype{P}[1]{>{\centering\arraybackslash}p{#1}}
\usepackage{graphicx}
\usepackage{array}
\usepackage{makecell}
\newcolumntype{x}[1]{>{\centering\arraybackslash}p{#1}}
\usepackage{tikz}

\usepackage{gensymb}
\newtheorem{theorem}{Theorem}

\newtheorem{remark}{Remark}

\usepackage{hyperref}
\usepackage{url}
\usepackage{enumerate}
\usepackage{graphicx}
\usepackage{epstopdf}
\usepackage{epstopdf} 

\usepackage{geometry}
\usepackage{cite}
\usepackage[nottoc]{tocbibind}
\usepackage{tikz}
\usetikzlibrary{shapes, shadows, arrows}
\tikzstyle{block}=[draw,rectangle,fill=blue!5,text width=12 em,text centered, minimum height=12mm, node distance=5 em]
\tikzstyle{line} = [draw,-latex']
\usetikzlibrary{calc,shapes.misc,shapes.geometric}

\begin{document}
	
		\title{Re-cycling  of DNA-containing Capsids Enhances Hepatitis B Virus Infection}
		
		
	\author{ Rupchand Sutradhar and D C Dalal}
	\date{} 
	\maketitle
		
	\section*{Abstract}
			 Hepatitis B virus (HBV) infection is a deadly liver disease. A part of the newly produced HBV DNA-containing capsids are reused as a core particle in HBV replication. It is investigated that the recycling of HBV capsids greatly affects the intercellular dynamics of HBV infection. The main purpose of the present work is to study the recycling effects of HBV DNA-containing capsids in the HBV infection.
		  Incorporating the  recycling effects of  capsids, a four-compartment  mathematical model is proposed for the first time in order to  understand the dynamics of HBV infection in a better way.
			The well-posedness of the model is obtained by showing  non-negativity, boundedness, and uniqueness of the solution.
			The explicit formula for the basic reproduction number is determined by applying the  next-generation approach method.  
			The proposed model  is solved with help of fourth-order explicit Runge-Kutta method.
			Through a rigorous comparison with experimental data obtained from four chimpanzees, the proposed model demonstrates a strong correspondence with the dynamics of HBV infection.
			A comprehensive global sensitivity analysis is also conducted to identify the most positively as well as negatively sensitive parameters for each compartment in the model.
		In the context of biology, depending on the value of basic reproduction number, the present  model possesses two  global asymptotically stable steady states: disease-free and endemic. The present study shows that the consideration of recycling of capsids reverses the existing mechanism of infection dynamics. 
			This study also reveals that  the accumulation of  capsids within the infected hepatocytes is a key factor for exacerbating the disease.
			Moreover, another major finding of our study is that due to recycling of capsids, the number of released viruses  increases in spite of low virus production rate.
			 The recycling of HBV DNA-containing capsids acts as a positive feedback loop in the viral infection.\vspace{0.2cm}\\
			Keywords: Hepatitis B; Capsids; Recycling effects; Stability; Lyapunov functional; Numerical simulation.
	\section*{Background}
	Hepatitis B virus (HBV) causes a deadly liver disease.  As reported by the World Health Organization (WHO), around 296 million people worldwide suffer from the chronic HBV infection, where 1.5 million new cases were found each year. It was estimated that 820 000 people died due to hepatitis B virus infection in the year 2019 alone, most of them had either cirrhosis or hepatocellular carcinoma (HCC) \cite{WHO_2021}. An important factor contributing to this large number of infected humans is the disease transmission rate. Naturally, there are two types of HBV infection: acute and chronic. In the acute phase of infection, HBV DNA copies may reach as high as $10^{10}$ copies/ml. Acute infection typically lasts a few weeks and are eventually cured as a result of immune response \cite{2001_Whalley,2007_Ciupe_Role}. As for the adult population, there is a clearance rate of 85-95\% in acute infection. However, chronic infections can last for many years and may result in diseases, such as liver cirrhosis and HCC \cite{2002_Ribeiro}. Chronic infection generally is a life-long incurable condition which affects the personal impacts of the patient such as stigma and discrimination, anxiety about disease progression, 
	and long-term health care costs etc.  According to the literature and clinical findings, horizontal and vertical transmission of this virus are the two main modes of transmission to human populations. Blood transfusions, unprotected sex, reusing syringes and blades, and the reuse of medical equipment during surgery are  examples of horizontal transmission.
	A complete understanding of the mechanisms of HBV persistence remains elusive. Consequently, HBV is a significant health concern for the  population of the World.
	
	The replication process of HBV is quite complex. In the course of infection, virion particles enter into the hepatocytes by the NA(+)-taurocholate co-transporting polypeptide (NTCP) receptor via endocytosis and uncoat themselves. The relaxed circular DNA (rcDNA) is delivered into the nucleus and is converted to covalently closed circular DNA (cccDNA) by the host DNA repair mechanism \cite{2002_lewin_hepatitis,2007_Guo_characterization}. RNA polymerase IIs (RNA Poly IIs) use the cccDNA as a template to produce viral RNAs, such as pgRNA, L, S, PreC, and X mRNAs and, the
	viral proteins are synthesized by ribosomes through the translation of mRNAs. Polymerase and pgRNA form a 1:1 complex called a Ribonucleoprotein (RNP) complex \cite{2021_Fatehi}. The nucleocapsid is formed by encapsulating this RNP complex by C proteins and it is commonly known as the immature nucleocapsid. The pgRNA is reverse-transcribed by polymerase \cite{2005_Murray}, resulting in the immature nucleocapsid being converted into a rcDNA-containing nucleocapsid, known as the mature nucleocapsid \cite{2021_Saraceni_review}. This newly produced rcDNA-containing nucleocapsid will either  return to the nucleus or be released as complete HBV
	from the infected hepatocytes \cite{2021_prifti}. A portion of  rcDNA-containing capsids is transported again into the nucleus to increase the amount of supercoiled cccDNA. It is known as recycling of rcDNA-containing capsids \cite{2021_Fatehi} and is considered as a significant factor responsible for the intracellular dynamics of HBV replication.

	In Literature, there are several viral  dynamic models \cite{2023_nayeem,2010_ji,2019_hui,2012_wang_mathematical,2015_Tchinda,2015_Moualeu} proposed during the last two decades. Some of these are associated with HBV infection. These models are useful in understanding the pathogenesis of infection as well as devising better treatment protocols. Nowak et al. \cite{1996_Nowak} analyzed a basic HBV infection model comprised of three compartments: uninfected hepatocytes, infected hepatocytes, and virions. Wodraz et al. \cite{2003_Wodarz} improved the basic model by including the effects of cytotoxic T cells and antibodies. By putting a standard incidence function in the place of mass action term of the uninfected hepatocytes and viruses, Min et al. \cite{2008_Min} and Chen et al. \cite{2014_Chen} modified the classic viral infection model \cite{1996_Nowak}. According to them the mass action term  is not rational  for the HBV infection as it implies that someone with a smaller liver is unlikely to be infected.
	In addition to the standard incidence function, the time delay during the production process of the virus was also taken into account by Gourley et al. \cite{2008_Gourley}.
	Rather than using constant growth terms, Hews et al. \cite{2010_Hews} introduced a modified model that considers the logistic growth of uninfected hepatocytes.
	Eikenberry et al. \cite{2009_Eikenberry} presented a delay model on HBV infection when uninfected hepatocytes proliferated logistically.
	Haung et al. \cite{2009_Huang} proposed a new model using another incidence function, called Beddington-DeAngelis type incidence function. Yu Ji et al. \cite{2010_YuJi} first showed that the effects of immune response is not constant, and it follows a periodic function. The cure rate of the infected hepatocytes is an important factor in viral infection. Wang et al. \cite{2010_Wang} extended the model of Min et al. \cite{2008_Min} by incorporating the effects of cure rate of infected hepatocytes. Fatehi et al. \cite{2018_fatehi_nkcell} determined that in HBV infection, NK cell takes a significant role in apoptosis as it kills the infected cell by producing the perforin and granzymes.  Using a mathematical model that incorporates intracellular components of HBV infection, Murray et al. \cite{2005_Murray} were able to estimate the infection dynamics and clearance of viruses in three chimpanzees who were acutely infected. Murray et al. \cite{2006_Murray} proposed another mathematical model with three compartments and  measured the half-life of HBV as  approximately four hours.
	This model was modified by considering the uninfected hepatocytes  and also studied with delay differential equation by Manna and Chakrabraty \cite{2015_Manna}. 
	Considering the effects of antibodies and CTLs, Danane et al. \cite{2018_Danane_optimal} extended the model proposed by Manna and Chakrabraty \cite{2015_Manna}, including CTL immune response and examined the optimality of the model. Using an average incidence rate, Guo et al. \cite{2018_Guo} established the global stability of a delayed-diffusive HBV infection model.  Fatehi et al. \cite{2021_Fatehi} built up an Intracellular model of HBV infection and compared various kinds of therapeutic strategies. An important biological indicator of virus dynamics is the age of the infected cells. Recently, Liu et al. \cite{2021_liu_age} proposed an age-structured model of HBV that treated HBV capsids as a separate compartment. There are mainly two roots of infection spread: virus-to-cell and cell-to-cell transmission. 
	
	In some recent biological studies \cite{2021_prifti,2021_Saraceni_review,2021_diogo_review}, it is observed that the severity of HBV infection are greatly affected by the recycling of HBV DNA-containing capsids but how this recycling of HBV capsids contributes to the virus replication is poorly understood. One of the  reason could be the increase in number  of viruses in the liver, since a portion of newly produced capsids that are eventually reused as core particles providing an additional source for supercoiled cccDNA.  
	In order to control HBV transmission in host, having a clear knowledge on recycling of capsids is extremely important. Although some mathematical studies have been conducted on HBV transmission in a host, only a few of them considered the HBV DNA-containing capsids as a separate compartment. However, these models failed   to capture the actual dynamics of HBV infection. The main reason for this failure could be ignoring the recycling effects of capsids. 
	In this study, an improved mathematical model incorporating  the recycling effects of HBV capsids is proposed for the   first time .
	This model is expected to reveal the HBV intracellular dynamics more realistically compared to already available models. In a nutshell, we  have mainly concetrated on the following things:
	\begin{enumerate}[(i)]
		\item The effects of recycling of capsids in the HBV infection.
		\item The effects of volume fraction of capsids on the disease dynamics.
		\item Global sensitivity analysis of model parameters using the method Latin hypercube sampling (LHS)-Partial rank correlation coefficient (PRCC).
		
	\end{enumerate}
	\section{Mathematical model}\label{Mathematical model}
	The persistence of HBV infection for a long period in patients depends on the stability of cccDNA in the infected hepatocytes. cccDNA plays a central role in disease progression. There are two main sources of cccDNA: HBV DNA-containing capsids produced directly from the incoming  virus from extracellular space and  HBV DNA-containing capsids produced within the hepatocytes by recycling \cite{2021_Fatehi,2021_Saraceni_review,2021_diogo_review}.  The recycling of capsids is not a continuous process. Depending on the availability of viral surface proteins (L, M, S), a portion of newly produced HBV DNA-containing capsids  goes back to the nucleus to amplify the pool of cccDNA. So, the volume fraction of capsids in favor of virus production is generally a function of the surface protein (L, M, S). Assume that $\alpha \in (0,1)$ be the volume fraction of capsids in favour of virus production, then $\alpha$ is a function of  surface proteins. However, a fixed estimated value of $\alpha$ is used throughout the study for simulation purposes. 
	Accordingly, the proposed model is described  by  the system of ordinary differential equations following:	
	\begin{align} \label{eq1}
		&\text{Susceptible hepatocytes}:~~~~~~~\dfrac{dX}{dt}=\lambda-\mu X-kVX,\nonumber\\
		&\text{Infected hepatocytes:}~~~~~~~~~~~~~\dfrac{dY}{dt}=kVX-\delta Y,\nonumber\\
		&\text{HBV DNA-containing capsids:}~\dfrac{dD}{dt}=aY+\gamma(1-\alpha)D-\alpha\beta D-\delta D,\\
		&\text{Viruses:}~~~~~~~~~~~~~~~~~~~~~~~~~~~~~~~\dfrac{dV}{dt}=\alpha \beta D-cV.\nonumber 
	\end{align} 
	Here, $X(t), Y(t), D(t)$ and $V(t)$ denote the numbers of susceptible hepatocytes, infected hepatocytes, HBV DNA-containing capsids, and free viruses respectively. All the parameters $\mu, k, \delta, a, \alpha, \beta, \gamma$ and $c$ are non-negative. In this model, $\lambda$ is assumed to be constant growth rate of susceptible hepatocytes, and $\mu$ is their natural death rate. The usual death rate of infected hepatocytes, as well as capsids is $\delta$. The parameter $k$ describes  rate at which the  susceptible hepatocytes are infected by the viruses. HBV capsids are produced at the rate $a$ from  infected hepatocytes , and $\beta$ denotes the production rate of new viruses. Here, $c$ is the virus clearance rate, and capsids reproduce themselves by recycling at rate $\gamma$. In Figure \ref{The_diagrammatic_representation}, the diagrammatic representation of the system \eqref{eq1} is shown.  
	\vspace{0.3cm}
		\begin{center}
		\begin{figure}[h!]
			\centering
			\includegraphics[width=15cm,height=10cm]{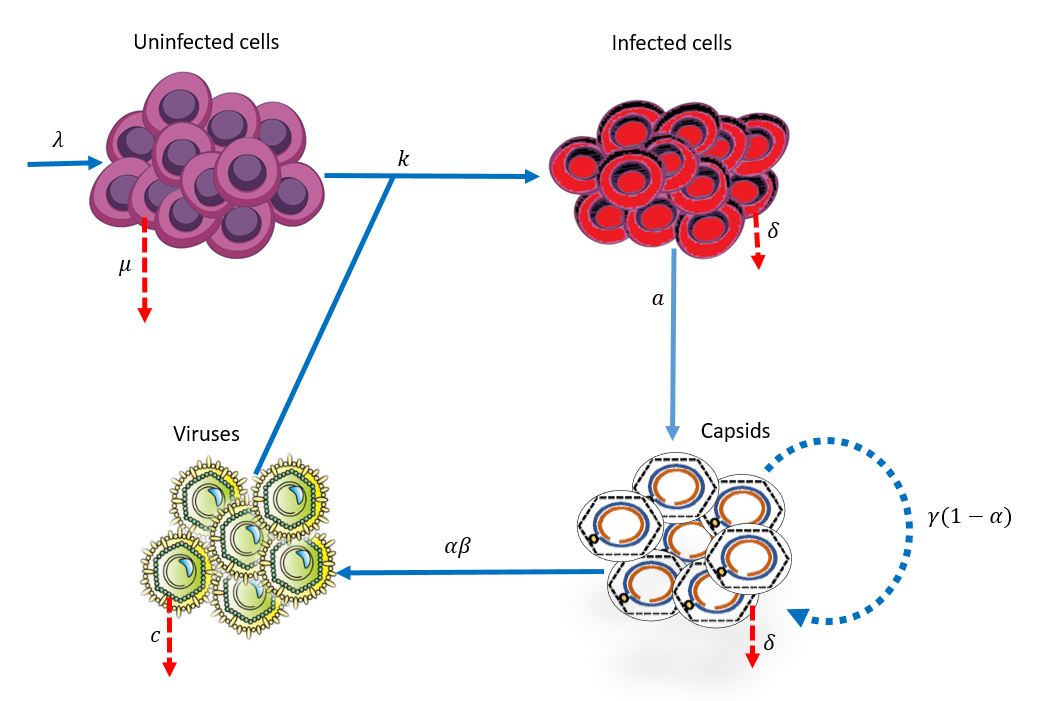}
			\caption{The diagrammatic representation of the system \eqref{eq1}.}
			\label{The_diagrammatic_representation}
		\end{figure}
	\end{center}

	\begin{table}[h]
		\footnotesize
		\begin{center}
			\begin{tabular}{|c|c|c|c|c|c|c|}
				\hline
				Uninfected  & Infected 		& HBV				 & Hepatitis B	 & Capsid-to-capsid  & Volume fraction  &References \\
				Hepatocytes & Hepatocytes 	& DNA-containing 	 &  Virus 			 & Production &   HBV DNA-containing     &  \\
				&            	& Capsids 	 		 &   		 & Rate &of capsids             &  \\ 
				\hline
				\textcolor{blue}{\checkmark} & \textcolor{blue}{\checkmark} & \textcolor{red}{\crossmark} & \textcolor{blue}{\checkmark} & \textcolor{red}{\crossmark} &\textcolor{red}{\crossmark}& \cite{1996_Nowak}\\ 
				\hline
				\textcolor{red}{\crossmark} & \textcolor{blue}{\checkmark} & \textcolor{blue}{\checkmark} & \textcolor{blue}{\checkmark} & \textcolor{red}{\crossmark} &\textcolor{red}{\crossmark}& \cite{2006_Murray}\\ 
				\hline
				\textcolor{blue}{\checkmark} & \textcolor{blue}{\checkmark} & \textcolor{blue}{\checkmark} & \textcolor{blue}{\checkmark} & \textcolor{red}{\crossmark} &\textcolor{red}{\crossmark}& \cite{2015_Manna,2018_Danane,2021_liu_age}\\ 
				\hline
				\textcolor{blue}{\checkmark} & \textcolor{blue}{\checkmark} & \textcolor{blue}{\checkmark} & \textcolor{blue}{\checkmark} & \textcolor{blue}{\checkmark}&\textcolor{blue}{\checkmark} & Present\\ 
				& & & &  && work\\
				\hline
			\end{tabular}
			\vspace{0.1cm}
			\caption{\label{Table-1} 
				Comparison between present work and some previously existing studies based on various factors.}
		\end{center}
	\end{table}
	\noindent
	This paper is organized as follows. In Section \ref{Solution property}, we prove  that our proposed model is well-posed by showing the existence, uniqueness, non-negativity, and boundedness of the solutions of the system \eqref{eq1}. In Section \ref{Stability of the solutions}, we calculate the basic reproduction number and determine the existence of all steady states. We also establish the local stability of both steady states using Routh Hurwitz criteria and global stability by constructing the appropriate Lyapunov functional and Lyapunov-LaSalle invariance principle in this section. In Section \ref{Numerical simulation}, we discuss some numerical results which support our analytical theory. In Section \ref{Comparison of the model}, the proposed model is compared to the existing one, and  the effects of some model parameters is analyzed by sensitivity analysis. We study the global sensitivity analysis of model parameters in section \ref{GSA}.  Finally, a brief conclusion is provided in Section \ref{Conclusions}.
	
	\section{Properties of the solutions } \label{Solution property}
	Here, the existence and uniqueness of the solution for the system \eqref{eq1} are discussed. In order to ensure the feasibility of the model from biological point of view, it is crucial to show that all solutions remain non-negative and bounded across all non-negative initial conditions. The rationale behind this is that the number of cells or viruses must not drop below zero or exhibit limitless growth after the time of infection. Accordingly, the subsequent discussion confirms the non-negativity and boundedness of the solution.
	
	\subsection{Existence and uniqueness of the solutions}
	Each function of the right hand side of system \eqref{eq1}  is polynomial functions of four variables $X(t),~Y(t),~D(t),~V(t)$. So, each function is   continuous and satisfies Lipschitz's condition on any closed interval $[0,\eta]$, $\eta\in \mathbb{R^+}$, set of all positive real numbers. Therefore, the solution of the system \eqref{eq1} exists and is unique. 
	\subsection{Non-negativity of the solutions}
	\begin{theorem}
		Under any non-negative initial conditions $(X(0)\geq 0,~Y(0)\geq 0,~D(0)\geq 0,~V(0)\geq 0)$, all solutions of the system \eqref{eq1} will remain non-negative.
	\end{theorem}
	\noindent
	The proof of this theorem is presented in Appendix A.
	\subsection{Boundedness of the solutions}
	\begin{theorem}
		For any non-negative initial condition $(X(0) \geq 0,~Y(0)\geq 0,~D(0)\geq 0,~V(0)\geq 0)$, the solutions of the system \eqref{eq1} are bounded for all $t>0$,  provided $R_s:=\alpha\beta-(1-\alpha)\gamma+\delta>0$ .
	\end{theorem}
	\begin{proof}
		In this section, it is shown that the solutions for any non-negative initial condition should be bounded. For this purpose,
		a new variable $T(t)$ is introduced as,		
		\begin{align*}
			T(t)=X(t)+Y(t).
		\end{align*}		
		\noindent
		Then, $\dfrac{dT}{dt}=\lambda-\mu X-\delta Y.$ Define $\rho=\min\{\mu, \delta\}$. Then we have,
		
		\begin{align*}
			\dfrac{dT}{dt}\leq \lambda-\rho(X+Y) \implies \dfrac{dT}{dt}\leq \lambda-\rho T,
		\end{align*}		
		\noindent
		which implies that
		$\lim \sup_{t\to\infty} T(t) \leq \dfrac{\lambda}{\rho}$. This shows that 	$\lim \sup_{t\to\infty} X(t) \leq \dfrac{\lambda}{\rho}$ and 	$\lim\sup_{t\to\infty}  Y(t) \leq \dfrac{\lambda}{\rho}$. So, $X(t)$ and $Y(t)$ are bounded for all $t>0$. Now, from the third equation of the system \eqref{eq1},  one can get 
		
		\begin{align*}
			\dfrac{dD}{dt}\leq \dfrac{a\lambda}{\rho}+(\gamma(1-\alpha)-\alpha\beta-\delta)D
			\implies\lim \sup_{t\to\infty}  D(t) \leq \dfrac{a \lambda}{\rho (\alpha\beta+\delta-\gamma(1-\alpha))},\\ ~\mbox{provided} ~ \alpha\beta-(1-\alpha)\gamma+\delta>0.\\
		\end{align*}		
		\noindent
		Therefore, $D(t)$ is bounded for all $t>0$. Using the boundedness of $D(t)$, from the last equation of the system \eqref{eq1} one can find	
		\begin{align*}
			&\dfrac{dV}{dt}\leq \dfrac{a \lambda\alpha\beta }{\rho (\alpha\beta+\delta-\gamma(1-\alpha))}-cV,  \\ &\implies \lim\sup_{t\to\infty}  V(t) \leq \dfrac{a \lambda \alpha\beta }{\rho c (\alpha\beta+\delta-\gamma(1-\alpha))}.
		\end{align*}
		
		\noindent	
		Hence, $V(t)$ is bounded for all $t>0$. Therefore, all the population ($X(t), Y(t), D(t) ~\& ~V(t)$) are bounded. Consequently, one can also observe the closed, bounded, and positively invariant set as follows:	
		\begin{align*}
			\mathcal{D}=&\Bigg\{\bigg(X(t), Y(t), D(t), V(t)\bigg)\in \mathbb{R}^{4}_{+}:~ 0\leq X(t)+Y(t)\leq \frac{\lambda}{\rho},\\
			&0\leq D(t)\leq \dfrac{a \lambda}{\rho (\alpha\beta+\delta-\gamma(1-\alpha))}, ~0\leq V(t)\leq \dfrac{a \lambda \alpha\beta }{\rho c (\alpha\beta+\delta-\gamma(1-\alpha))}\Bigg\}
		\end{align*}		   
	\end{proof}
	\begin{remark}
		The condition $R_s>0\implies \alpha\beta-(1-\alpha)\gamma+\delta>0\implies \gamma<\frac{\alpha\beta+\delta}{1-\alpha}.$ If $\gamma$ meets this condition, then $D(t)$ becomes always bounded. Otherwise, $D(t)$ diverges to positive infinite \textit{i.e.} the severity of infection increases significantly and situation of the patient becomes worse and worse with time. This relationship between these four parameters is very important when treating HBV patients.
	\end{remark}
    \section*{Results}
	\section{Existence and stability of  equilibria} \label{Stability of the solutions}	
	Before proceeding to the detailed study, it is mentioned that the condition $R_s>0$ will be used throughout the further  study. 		 
	\subsection{Existence of equilibria}
	In order to evaluate the equilibrium points of the system \eqref{eq1}, one need to consider the zero growth isoclines and their points of intersection. Thus, the equilibrium points of the system \eqref{eq1} are found by solving the system of equations	
	\begin{align} \label{eq1_equilibrium}
		&\lambda-\mu X-kVX=0,\nonumber\\
		&kVX-\delta Y=0,\nonumber\\
		&aY+\gamma(1-\alpha)D-\alpha\beta D-\delta D=0,\\
		&\alpha \beta D-cV=0.\nonumber 
	\end{align} 
	It can be shown that the system \eqref{eq1} possesses two sets of equilibrium points. 
	\begin{enumerate}[(i)]
		\item \textbf{The uninfected steady-state or disease-free equilibrium ($E_u$):} The uninfected steady-state or disease-free equilibrium always exists, and it is denoted by $E_u=\left(\dfrac{\lambda}{\mu},0,0,0\right)$.
		\item \textbf{The infected steady-state or endemic equilibrium ($E_i$):} Mathematically, infected steady-state or endemic equilibrium points exists always. But biologically,  the existence of this steady state depends on the basic reproduction number ($R_0$) and $R_s$ (which is defined earlier). When $R_0>1$, and $R_s>0$, the endemic equilibrium ($E_i$) occurs  and is given by $E_i=\left(X_1,Y_1,D_1,V_1\right)$, where
		\begin{align*}
			X_1&=\frac{c \delta  (\alpha  (\beta +\gamma )-\gamma +\delta )}{a \alpha  \beta  k}, 
			~~~~~~~~~~~~~~~Y_1=\frac{a \alpha  \beta  k \lambda -c \delta  \mu  (\alpha  (\beta +\gamma )-\gamma +\delta )}{a \alpha  \beta  \delta  k},\\
			D_1&=\frac{a \alpha  \beta  k \lambda -c \delta  \mu  (\alpha  (\beta +\gamma )-\gamma +\delta )}{\alpha  \beta  \delta  k (\alpha  (\beta +\gamma )-\gamma +\delta )},
			~V_1=\frac{a \alpha  \beta  k \lambda -c \delta  \mu  (\alpha  (\beta +\gamma )-\gamma +\delta )}{c \delta  k (\alpha  (\beta +\gamma )-\gamma +\delta )}.
		\end{align*}
	\end{enumerate}
	\subsection{Basic reproduction number ($R_0$)}
	In case of viral infection, the basic reproduction number is the number of secondary infective cells produced by a single infective cell, which is introduced into a fully susceptible population \cite{1990_Diekmann}.
	The first equation of the system \eqref{eq1} is for the uninfected class and the last three equations are meant for the infected class. The next-generation approach \cite{2002_Driessche_next_g_m,2005_Heffernan_next_g_m} is used to determine the basic reproduction number. Rewrite the last three equations of system \eqref{eq1} as, 	
	\begin{align} \label{system_R0}
		&\dfrac{dY}{dt}=\mathcal{F}_1-\mathcal{V}_1~ \mbox{where}~ \mathcal{F}_1=kVX~ \mbox{and}~ \mathcal{V}_1=\delta Y, \nonumber\\
		&\dfrac{dD}{dt}=\mathcal{F}_2-\mathcal{V}_2~ \mbox{where}~ \mathcal{F}_2=0 ~\mbox{and}~ \mathcal{V}_2=-aY-\gamma(1-\alpha)D+\alpha\beta D+\delta D,\\
		&\dfrac{dV}{dt}=\mathcal{F}_3-\mathcal{V}_3~ \mbox{where}~ \mathcal{F}_3=0 ~\mbox{and}~ \mathcal{V}_3=-\alpha\beta D+cV.\nonumber
	\end{align}	
	\noindent
	We define
	\begin{align*} \mathcal{F}=
		\begin{bmatrix}
			~\dfrac{\partial \mathcal{F}_1}{\partial Y}&~ \dfrac{\partial \mathcal{F}_1}{\partial D} & ~\dfrac{\partial \mathcal{F}_1}{\partial V}~ \\[12 pt]
			\dfrac{\partial \mathcal{F}_2}{\partial Y} & \dfrac{\partial \mathcal{F}_2}{\partial D} & \dfrac{\partial \mathcal{F}_2}{\partial V} \\[12 pt]
			\dfrac{\partial \mathcal{F}_3}{\partial Y} & \dfrac{\partial \mathcal{F}_3}{\partial D} & \dfrac{\partial \mathcal{F}_3}{\partial V} \\
		\end{bmatrix}_{E_u} ~~\mbox{and}~~~
		\mathcal{V}=
		\begin{bmatrix} 
			~\dfrac{\partial \mathcal{V}_1}{\partial Y}& ~\dfrac{\partial \mathcal{V}_1}{\partial D} & ~\dfrac{\partial \mathcal{V}_1}{\partial V}~ \\[12 pt]
			\dfrac{\partial \mathcal{V}_2}{\partial Y} & \dfrac{\partial \mathcal{V}_2}{\partial D} & \dfrac{\partial \mathcal{V}_2}{\partial V} \\[12 pt]
			\dfrac{\partial \mathcal{V}_3}{\partial Y} & \dfrac{\partial \mathcal{V}_3}{\partial D} & \dfrac{\partial \mathcal{V}_3}{\partial V} \\
		\end{bmatrix}_{E_u},
	\end{align*}
	\noindent
	where $E_u$ denotes the disease-free equilibrium point which is calculated above. After simplification, we get	
	\begin{align*} \mathcal{F}=
		\begin{bmatrix}
			~0~ & 0~ & \dfrac{k\lambda}{\mu}~ \\
			0 & 0 & 0 \\
			0 & 0 & 0 \\
		\end{bmatrix}~~
		\mbox{and}~~ \mathcal{V}=
		\begin{bmatrix}
			\delta & 0 & 0 \\
			-a & -\gamma(1-\alpha)+\alpha\beta +\delta  & 0 \\
			0 & -\alpha\beta & c 
		\end{bmatrix}.
	\end{align*}
	The basic reproduction number of the system \eqref{eq1} is defined by $R_0=\rho\left(\mathcal{F} \mathcal{V}^{-1}\right)$, where $\rho\left(\mathcal{F} \mathcal{V}^{-1}\right)$ is defined as the spectral radius of the matrix $\mathcal{F} \mathcal{V}^{-1}$, which is given by 
	$R_0=\dfrac{ak\lambda\alpha\beta}{\left(c\alpha\beta\delta-c\gamma\delta+c\alpha\gamma\delta+c\delta^2\right)\mu}$.
	\newline

	\subsection{Stability analysis of equilibria}
	
	\subsubsection{Local stability analysis}
	%
	\begin{theorem}
		The disease-free equilibrium $E_u$ is locally asymptotically stable when  $R_0<1$ and $R_s>0$.
	\end{theorem}
	\begin{proof}
		The Jacobian matrix of the system \eqref{eq1} at the disease-free steady state $(E_u)$ is given by	
		\begin{align}
			\begin{bmatrix} \label{Jacobian_disease free}
				-\mu & 0 & 0	& -\dfrac{k\lambda}{\mu}\\
				0 & -\delta & 0 & \dfrac{k\lambda}{\mu}\\
				0 & a & \gamma ( 1-\alpha )-\alpha\beta-\delta & 0\\
				0 & 0 & \alpha\beta & -c
			\end{bmatrix}.
		\end{align}	
		\noindent
		It is clear that $-\mu$ is a negative eigenvalue of the Jacobian matrix \eqref{Jacobian_disease free}. The other three eigenvalues of the Jacobian matrix \eqref{Jacobian_disease free} are the eigenvalues of the matrix	
		\begin{align}
			\begin{bmatrix}\label{Jacobian_disease free 3 cross 3}
				-\delta & 0 & \dfrac{k\lambda}{\mu}\\
				a & \gamma ( 1-\alpha )-\alpha\beta-\delta & 0\\
				0 & \alpha\beta & -c
			\end{bmatrix}.
		\end{align}		
		\noindent
		By using the Routh-Hurwitz criteria \cite{2015_martcheva_introduction}, it is  shown that all eigenvalues of the matrix \eqref{Jacobian_disease free 3 cross 3} are either negative or have  negative real parts. Let the characteristic equation of this matrix \eqref{Jacobian_disease free 3 cross 3} be	
		\begin{align*}		
			x^3+A_1 x^2+A_2 x+A_3=0.
		\end{align*}		
		\noindent
		When  $R_0<1$ and $R_s>0$,		
		\begin{align*}
			A_1&=R_s+c+\delta>0,\\
			A_2&=(\delta+c)R_s+c\delta>0,\\
			A_3&=-\dfrac{a \alpha  \beta  k \lambda }{\mu }+\alpha  \beta  c \delta -(1-\alpha ) c \gamma  \delta +c \delta ^2,\\
			& =-\dfrac{a \alpha  \beta  k \lambda }{\mu }+c\delta R_s>0,\\
			A_1A_2-A_3&=(R_s+c+\delta)((\delta+c)R_s+c\delta)+\dfrac{a \alpha  \beta  k \lambda }{\mu }-c\delta R_s,\\
			&=R_s(R_s+c+\delta)(c+\delta)+c\delta(c+\delta)+\dfrac{a \alpha  \beta  k \lambda }{\mu }>0.\\
		\end{align*}		
		\noindent	
		Thus, the Routh-Hurwitz criteria is satisfied. Therefore, local asymptotically stability of the disease-free equilibrium is established when $R_0<1.$ In case of  $R_0=1$, it is seen that the determinant of  Jacobian matrix \eqref{Jacobian_disease free} become zero and consequently it has at least one zero eigenvalue, and also when $R_0>1$, the matrix \eqref{Jacobian_disease free} has at least one positive eigenvalue and thus the disease-free equilibrium point will be unstable.
	\end{proof} 
	\begin{theorem}
		The endemic equilibrium $E_i$ is locally asymptotically stable when the basic reproduction number $R_0>1$ and does not exist if $R_0<1$.
	\end{theorem}
	\begin{proof}
		The Jacobian matrix at the equilibrium $E_i$ is		
		\begin{align*}
			\begin{bmatrix}
				-\mu-kV_1 & 0 & 0	& -kX_1\\
				kV_1 & -\delta & 0 & kX_1\\
				0 & a & \gamma ( 1-\alpha )-\alpha\beta-\delta & 0\\
				0 & 0 & \alpha\beta & -c
			\end{bmatrix}.
		\end{align*}	
		\noindent
		The characteristic equation of the Jacobian matrix is given by		
		\begin{align*}
			(-c-x) \left(\delta  \mu +\delta  k V_1+k V_1 x+x^2+\delta  x+\mu  x\right) (-\alpha  \beta -\alpha  \gamma +\gamma -\delta -x)\\-\alpha  \beta  (a k x X_1+a k \mu  X_1)=0.\\
		\end{align*}		
		\noindent
		Comparing above equation with this equation	$x^4+B_1 x^3+B_2 x^2+B_3 x+B_4=0$, we have		
		\begin{align*}
			&B_1 =~ R_s+(c+\delta +k V_1+\mu)>0,\\
			&B_2 =~ R_s(\delta+kV_1+\mu+c)+kV_1(c+\delta)+\delta(c+\mu)+\mu>0,\\			
			&B_3 = ~\frac{ak\lambda\alpha\beta(R_s(c+\delta)+c\delta)}{R_s c\delta}>0,\\			
			&B_4 = ~a \alpha  \beta  k \lambda -R_s c \delta  \mu  >0,\\			
			&B_1B_2-B_3 =~  k^2V_1^2(R_s+\delta)+(R_s+\delta)(R_s+\delta+\mu)(\delta+\mu)+c^2(R_s+kV_1+\delta+\mu)\\&+c(R_s+kV_1+\delta+\mu)+k(aX_1\alpha\beta+V_1(R_s+\delta)(R_s+\delta+2\mu))>0,\\			
			&B_1B_2B_3-B_3^2-B_1^2B_4 = ~a^3 \alpha ^3 \beta ^3 k^3 \lambda ^3 (R_s+c) (R_s+\delta ) (c+\delta )+R_s^4 c^4 \delta ^4 \mu  (R_s+c+\delta )^2\\
			&+a \alpha  R_s^2 \beta  c^2 \delta ^2 k \lambda  (R_s+c+\delta )\Big(R_s^2 \left(c^2+c \delta +\delta ^2\right)+R_s c \delta  (c+\delta +2 \mu )+c^2 \delta ^2\Big)\\
			&+a^2 \alpha ^2 R_s \beta ^2 c \delta  k^2 \lambda ^2 \Big(R_s^3 (c+\delta )+R_s^2(2 c^2+3 c \delta +2 \delta ^2)+R_s (c \delta  \mu +(c+\delta )^3)\\&+c \delta  (c+\delta )^2\Big)>0.		
		\end{align*}
		
		\noindent
		So, $B_1$, $B_2$, $B_3$, and $B_4$ satisfy all the conditions of Routh-Hurwitz criteria when $R_0 > 1$ and $R_s>0$. Thus, all the eigenvalues of the matrix are either negative or have negative real parts. Therefore, the endemic equilibrium is locally asymptotically stable. 
	\end{proof}

	\subsubsection{Global stability analysis}
	In order to prove the global stability of equilibria, Theorem 7.1 of the book \cite{2015_martcheva_introduction} is used.  Two suitable Lyapunuv functions are defined as follows:
	\begin{align*}
		\mathcal{L}_1(t)&=X_0\left(\dfrac{X(t)}{X_0}-1-\ln\left (\dfrac{X(t)}{X_0}\right )\right )+Y(t)+\dfrac{\delta}{a}D(t)+\dfrac{\delta(-\gamma(1-\alpha)+\alpha\beta+\delta)}{a\alpha\beta}V(t).\\
		\mathcal{L}_2(t)&=X_1\left(\frac{X(t)}{X_1}-1-\ln\left (\frac{X(t)}{X_1}\right )\right)+Y_1\left(\frac{Y(t)}{Y_1}-1-\ln\left (\frac{Y(t)}{Y_1}\right )\right)\nonumber\\
		&+\frac{\delta D_1}{a}\left(\frac{D(t)}{D_1}-1-\ln\left (\frac{D(t)}{D_1}\right )\right)+\frac{\delta R_s V_1}{a\alpha\beta}\left(\frac{V(t)}{V_1}-1-\ln\left (\frac{V(t)}{V_1}\right )\right), 
	\end{align*}
	Both functions are radially unbounded and globally positively definite. Therefore, the choices of Lyapunov functions are appropriate.
	\begin{theorem}\label{Theorem 5}
		The disease-free equilibrium point $E_u$ is globally asymptotically stable if  $R_0\leq 1$.
	\end{theorem}
	\begin{proof}
		To prove the asymptotic global stability of disease-free equilibrium point, the first Lyapunov function $\mathcal{L}_1(t)$ (defined above) is considered.
		\begin{align*}
			\mathcal{L}_1(t)=X_0\left(\frac{X(t)}{X_0}-1-\ln\left (\frac{X(t)}{X_0}\right )\right )+Y(t)+\frac{\delta}{a}D(t)+\frac{\delta(-\gamma(1-\alpha)+\alpha\beta+\delta)}{a\alpha\beta}V(t).	
		\end{align*}		
		\noindent
		Now,		
		\begin{align*}
			\frac{d\mathcal{L}_1(t)}{dt}&=X_0\left (\frac{X'}{X_0}-\frac{X'}{X}\right)+Y'(t)+\frac{\delta}{a}D'(t)+\frac{\delta(-\gamma(1-\alpha)+\alpha\beta+\delta)}{a\alpha\beta}V'(t)\\
			&=\left(1-\frac{X_0}{X}\right)X'(t)+Y'(t)+\frac{\delta}{a}D'(t)
			+\frac{\delta(-\gamma(1-\alpha)+\alpha\beta+\delta)}{a\alpha\beta}V'(t)\\
			&=\left(1-\frac{X_0}{X(t)}\right)(\lambda-\mu X(t)-kV(t)X(t))+kV(t)X(t)\\
			&+\frac{\delta}{a} \left(aY(t)+\gamma(1-\alpha)-\alpha\beta-\delta D \right)
			+\frac{\delta(-\gamma(1-\alpha)+\alpha\beta+\delta)}{a\beta}(\beta D-cV)\\
			&=\lambda\left(2-\frac{X}{X_0}-\frac{X_0}{X}\right)+\frac{a\alpha\beta k\lambda-c\delta\mu(-\gamma(1-\alpha)+\alpha\beta+\delta)}{a\alpha\beta\mu}V.
		\end{align*}		
		\noindent
		From the article of Kajiwara et al.\cite{2012_Kajiwara}, $\left(2-\dfrac{X}{X_0}-\dfrac{X_0}{X}\right)\leq 0$ and since, $R_0\leq 1$, $a\alpha\beta k\lambda-c\delta\mu(-\gamma(1-\alpha)+\alpha\beta+\delta)\leq 0$. Therefore, $ \dfrac{{d\mathcal{L}}_1(t)}{dt}\leq 0$. We consider the largest invariant set $\mathcal{M}_1=\left\{(X,Y,D,V) \in \mathbb{R}^4:~\dfrac{d\mathcal{L}_1(t)}{dt}=0\right\}$ . 
		The solution of the equation $\dfrac{d\mathcal{L}_1(t)}{dt}=0$ is  $\left(\dfrac{\lambda}{\mu},0,0,0\right)$ only, which is the equilibrium point $E_u$. So, based on the Lyapunov-LaSalle invariance principle \cite{2015_martcheva_introduction}, the disease-free equilibrium point, $E_u$ is globally asymptotically stable whenever $R_0\leq 1$.	
	\end{proof}
	\begin{theorem}\label{Theorem 6}
		The endemic equilibrium $E_i$ is globally asymptotically stable if $R_0>1$.
	\end{theorem}
	\begin{proof}
		We are approaching the problem by taking into account the second Lyapunov function $\mathcal{L}_2(t)$.	
		\begin{align} \label{lyapunov_2}
			\mathcal{L}_2(t)&=X_1\left(\frac{X(t)}{X_1}-1-\ln\left (\frac{X(t)}{X_1}\right )\right)+Y_1\left(\frac{Y(t)}{Y_1}-1-\ln\left (\frac{Y(t)}{Y_1}\right )\right)\nonumber\\
			&+\frac{\delta D_1}{a}\left(\frac{D(t)}{D_1}-1-\ln\left (\frac{D(t)}{D_1}\right )\right)+\frac{\delta R_s V_1}{a\alpha\beta}\left(\frac{V(t)}{V_1}-1-\ln\left (\frac{V(t)}{V_1}\right )\right), 
		\end{align}		
		\noindent
		where $R_s$ is defined above. \\Differentiating \eqref{lyapunov_2} with respect to $t$, we have		
		\begin{align*}
			\frac{d\mathcal{L}_2(t)}{dt}&=X_1\left (\frac{X'(t)}{X_1}-\frac{X'(t)}{X(t)}\right)+Y_1\left (\frac{Y'(t)}{Y_1}-\frac{Y'(t)}{Y(t)}\right)\\
			&+\frac{\delta D_1}{a}\left (\frac{D'(t)}{D_1}-\frac{D'(t)}{D(t)}\right)+\frac{\delta R_s V_1}{a\alpha\beta}\left (\frac{V'(t)}{V_1}-\frac{V'(t)}{V(t)}\right)\\
			&=\left(1-\frac{X_1}{X(t)}\right)X'(t)+\left(1-\frac{Y_1}{Y(t)}\right)Y'(t)+\frac{\delta D_1}{a}\left(1-\frac{D_1}{D(t)}\right)D'(t)\\
			&+\frac{\delta R_s V_1}{a\alpha\beta}\left(1-\frac{V_1}{V(t)}\right)V'(t)\\
			&=-\mu X(t)\left(1-\frac{X_1}{X(t)}\right)^2+\delta Y_1\left ( 4-\frac{Y(t)D_1}{D(t)Y_1}-\frac{D(t)V_1}{D_1V(t)}-\frac{X(t)Y_1V(t)}{X_1Y(t)V_1}-\frac{X_1}{X(t)}\right).
		\end{align*}	
		\noindent
		\begin{flushright}
			$\left[\text{Using}~ D_1=\dfrac{a}{R_s}Y_1 ~\text{and}~ V_1=\dfrac{a\alpha\beta}{R_s c}Y_1\right]$
		\end{flushright}
		Clearly, the first term of the above equation is negative unless $X(t)=X_1$. To prove the negativity of the second  term, 
		define $$~x_1=\dfrac{Y(t)D_1}{D(t)Y_1},~x_2=\dfrac{D(t)V_1}{D_1V(t)},~ x_3=\dfrac{X(t)Y_1V(t)}{X_1Y(t)V_1},~x_4=\dfrac{X_1}{X(t)}.$$ It is clear that $x_i\geq 0$ for $i=1,2,3,4$ and $x_1x_2x_3x_4=1$. Applying  $A.M \geq G.M$, 		
		\begin{align*}
			&\dfrac{x_1+x_2+x_3+x_4}{4}\geq(x_1x_2x_3x_4)^{\frac{1}{4}},\\
			\implies& \dfrac{Y(t)D_1}{D(t)Y_1}+\dfrac{D(t)V_1}{D_1V(t)}+\dfrac{X(t)Y_1V(t)}{X_1Y(t)V_1}+\dfrac{X_1}{X(t)}\geq 4.
		\end{align*}	
		\noindent
		Hence, $\dfrac{{d\mathcal{L}}_2(t)}{dt}\leq0$ when $R_0>1$. Let the largest invariant set $\mathcal{M}_2=\left\{\left(X,Y,D,V\right) \in \mathbb{R}^4:\dfrac{{d\mathcal{L}}_2(t)}{dt}=0\right\}$.
		It is noticed that the solution of the equation $\dfrac{d\mathcal{L}_2(t)}{dt}=0$ is only
		$(X_1,Y_1,D_1,V_1)$, which is the equilibrium point $E_i$. So, the endemic equilibrium point is globally asymptotically stable whenever $R_0> 1$ based on the Lyapunov-LaSalle invariance principle.	
	\end{proof}
	\section{Bifurcation Analysis}
	In this section, the bifurcation analysis of the proposed model \eqref{eq1} is performed. Stability criteria of $E_0$ indicates that $E_0$ will be stable if $\mu>\dfrac{ak\alpha\beta\lambda}{R_sc\delta}(=\mu^*)$, otherwise it will be unstable \textit{i.e.}  the stability of $E_0$ changes as $\mu$ crosses the threshold value $\mu=\mu^*$. On the other hand, the endemic equilibrium $E_i$ (although $E_i$ is not feasible in the context of biology) becomes unstable if $\mu>\mu^*$. Thus, the equilibrium points  coincide and exchange their stability which leads to the transcrical bifucation of the system \eqref{eq1} around the point $E_0$ at $\mu=\mu^*$ with natural death rate parameter $\mu$ as the bifurcation parameter. It is seen that the Jacobian matrix \eqref{Jacobian_disease free} has one zero eigenvalue when $\mu=\mu^*$. 
	In order to verify the existence of transcritical bifurcation analytically, Sotomayor’s theorem \cite{2013_Perko_differential} is applied on the system \eqref{eq1} around the disease-free equilibrium point $E_0$ when $\mu=\mu^*$.
	\noindent	
	The R.H.S of system \eqref{eq1} can be represented in vector form as
	\begin{align}
		f(X,Y,D,V)=
		\begin{pmatrix} \label{f function bifurcation}
			\lambda-\mu X-kVX\\
			kVX-\delta Y\\
			aY+\gamma(1-\alpha)D-\alpha\beta D-\delta D\\
			\alpha \beta D-cV
		\end{pmatrix}
	\end{align}
	Differentiating  the function $f(X,Y,D,V)$ partially with respect to $\mu$, it is obtained that $f_{\mu}(X,Y,D,V)=(-X~~0~~0~~0)^T$.
	The Jacobian matrix \eqref{Jacobian_disease free} at disease-free equilibrium point ($E_0$) and its transpose matrix have an eigenvalue $\xi=0$ with eigenvectors $$v=\left(\dfrac{k\lambda}{\mu^2}~~\dfrac{a\lambda}{\mu \delta}~~\dfrac{ak\lambda}{R_s\mu\delta}~~1\right)^T
	\mbox{and}~ 
	w=\left(0~~\dfrac{a\alpha\beta}{\delta R_s}~~\dfrac{\alpha\beta}{R_s}~~1\right)
	^T.$$\\
	Now, at $E_0$ we have verified the following transvesality conditions:
	\begin{enumerate}[(i)]
		\item $w^T f_{\mu}(E_0,\mu^*)=\left(0~~\dfrac{a\alpha\beta}{\delta R_s}~~\dfrac{\alpha\beta}{R_s}~~1\right)\left(-\dfrac{\lambda}{\mu}~0~0~0\right)^T=0$
		\item $w^T[Df_{\mu}(E_0,\mu^*)v]=-\dfrac{k\lambda}{\mu^2}\neq 0$
		\item $w^T[D^2f(E_0,\mu^*)(v,v)]=-\dfrac{2\alpha\beta k^2 \lambda}{R_s \delta \mu^2}\neq 0$
	\end{enumerate}
	All the notations used here are same as in the book of  Lawrence Perko \cite{2013_Perko_differential}. Therefore, all three conditions of Sotomayor’s theorem hold and the system \eqref{eq1} undergoes transcritical bifurcation at $E_0$ as the natural death rate of uninfected hepatocytes $\mu$, crosses the threshold value $\mu^*$.

	\begin{table}[h]
		\begin{center}	
			\caption{\label{Table-2}  Estimation of parameters.}		
			\begin{tabular}{||c||l||c||c||}
				\hline
				Parameters  & ~~~~~Descriptions  	 Values 		& Estimated value	 &Units \\ [.5ex]
				
				&&(Baseline value)&\\
				\hline
				$\lambda$ 	& Constant growth rate of & $2.67\times10^7$   & $\mbox{cells  ml}^{-1}\mbox{day}^{-1}$  \\
				
				& uninfected hepatocytes& & \\
				\hline
				$\mu$		& Natural death rate of 			& $0.096$      & $\mbox{day}^{-1}$   \\
				
				& uninfected hepatocyte& & \\
				\hline
				$k$			& Virus-to-cell infection rate				&$3.38\times10^{-12}$&$\mbox{ml virion}^{-1}\mbox{day}^{-1}$    						\\
				\hline
				$a$			& Production rate of capsid  & 157     &$\mbox{capsids  cell}^{-1}\mbox{day}^{-1}$    \\
				
				& from infected hepatocyte	&      &    	\\
				\hline
				$\beta$ 	&  Production rate of virus 			& 1.83    &  $\mbox{day}^{-1}$   	\\
				
				& from capsid& & \\
				\hline
				$\delta$	& Death rate of infected   & 0.24     &  $\mbox{day}^{-1}$    \\
				&  hepatocyte \& capsid &      &    \\ 
				\hline
				$c$ 		& Death rate of virus					& 3.93     &  $\mbox{day}^{-1}$   \\
				\hline
				$\alpha$ 	& Volume fraction of HBV  			&0.84      & unitless    \\
				&   DNA-containing capsid 			&    &   \\
				\hline
				$\gamma$ 	&  Capsid to capsid production rate 		&  1.24     &$\mbox{day}^{-1}$  \\
				& recycling rate of capsids& & \\
				\hline
			\end{tabular}
		\end{center}
	\end{table}	
	\section{Parameter estimation and model calibration}
	In this section, in order to enhance the realism of viral dynamics and to improve the reliability of our robust predictions, the proposed model is  calibrated  using  experimental data collected from the previously published article of Asabe et al.\cite{2009_Asabe}. In their research, Asabe et al. examined the effects of varying inoculum sizes on the kinetics of viral spread and immune system  in a cohort of nine young, healthy chimpanzees. In adherence to ethical guidelines,  handling procedures of all animals were carried out with the sole intention of human use. It was observed that the viral load  for six out of nine  chimpanzees reached a peak level of $2\times10^{10}$ within three weeks from the time of inoculation. Later, the virion load decreased within 15 weeks and reached  below the detection level. On the other hand, this experiment also documented that the remaining three chimpanzees developed chronic infection. The concentration of HBV DNA of these three chimpanzees (1603, 1616, A2A007) who persisted infection and another chimpanzee (1618) which developed acute infection are considered in order to estimate the model parameters and  validate our model. In each case, the model parameters are estimated by minimizing the sum-of-squares error (SSE) which is given by 
	\begin{align}
		SSE=\sum_{i=1}^{n}(P_i-p(i))^2, 
	\end{align} 
	where $P_i$ and $p(i)$ denote the experimental data and model solution, respectively. Four sets of parameters' values are obtained for four chimpanzees. For further analysis of the system \eqref{eq1} and numerical simulation, the average value of each parameter will be used throughout the study. The average values of parameters are given in Table \ref{Table-2}. In addition, the estimated parameter  values  are  compared with the existing values in the literature \cite{2007_Dahari,2006_Murray}, and it is seen that there is no significant disparities between the two sets of values. 
	
	In Figure \ref{AllChimpanzee}, the experimental data and solution of the proposed  model are compared. It is observed that for all four chimpanzees, the proposed model effectively captures the infection dynamics. Although the model solution agrees well with the experimental data, there are some discrepancies. The possible reasons for this could be that many crucial factors of virus dynamics such as the roles of immune system, intracellular delay are neglected in the proposed model. 
	\begin{center}
		\begin{figure}[h!]
			\centering
			\includegraphics[width=15cm,height=10cm]{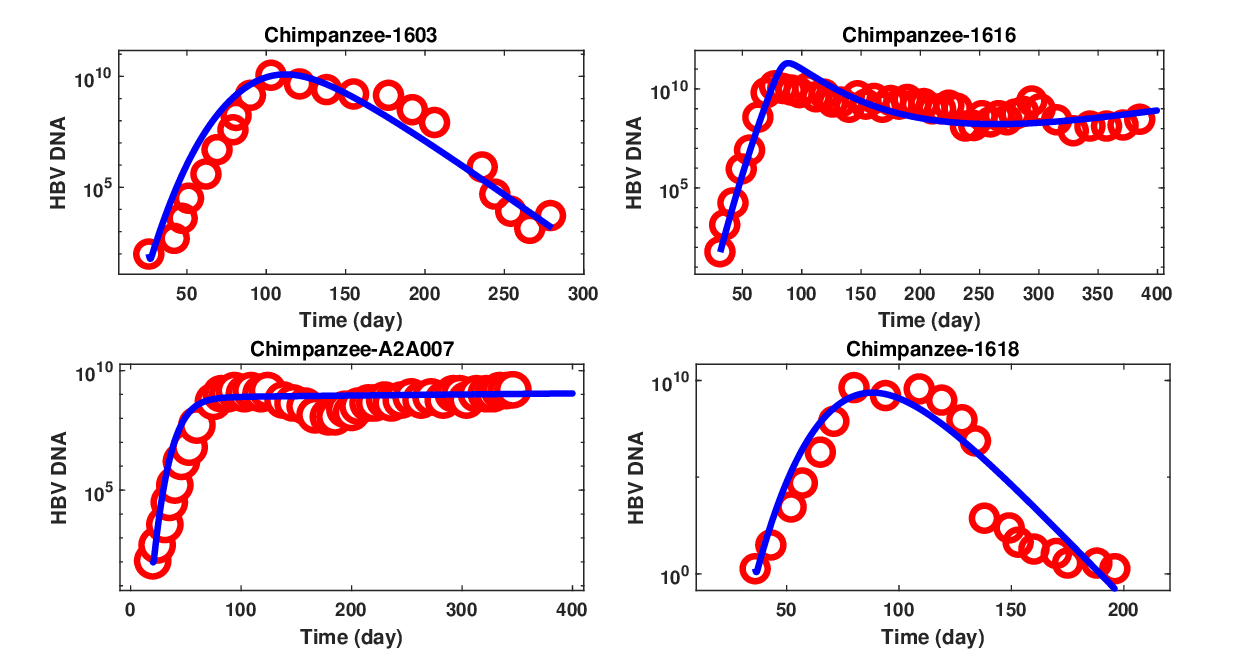}
			\caption{Experimental validation of the model. The experimental data is represented by red circles, while the solid blue line corresponds to the numerical solution of the system \eqref{eq1}.}
			\label{AllChimpanzee}
		\end{figure}
	\end{center}
	\section*{Discussion}
	\section{Numerical simulation} \label{Numerical simulation}
	From our theoretical results it is seen that $R_0$ and $R_s$ are two crucial threshold numbers for the dynamics of HBV infection. Here, we numerically study the stability of the model for different values of $R_0$. The numerical simulation examines the asymptotic behavior of the system \eqref{eq1} for different cases. The system \eqref{eq1} is solved numerically 
	and the results are plotted graphically. We utilize the value of parameters from Table \ref{Table-2}.
	In this context, we consider three different initial conditions, as $ic_1=(2.56\times10^{8}, 0.99\times10^{8}, 1.60\times 10^{10}, 0.369\times 10^{10})$, $ic_2=(7.68\times10^{8}, 2.97\times10^{8}, 4.82\times 10^{10}, 1.10\times 10^{10})$ \mbox{and} $ic_3=(12.79\times10^{8} 4.95\times10^{8}, 8.04\times 10^{10}, 1.85\times 10^{10})$. In addition,  two sets of parameters are chosen, the first one provides $R_0<1$, and the second one leads to $R_0>1$.
	\subsection{Disease free equilibrium }
	In order to examine the disease dynamics for $R_0<1$, the value of $k$ is different from Table \ref{Table-2}, namely $k=3\times10^{-13}$ \cite{2007_Ciupe}, is chosen while the values of other parameters remain unchanged. The resulting dynamics is represented in Figure \ref{fig_1}. Consider the duration of the simulation as 500 days. For each initial condition, the concentration level of uninfected hepatocytes increases gradually, and the concentration level of infected hepatocytes decreases with time before stabilizing to the disease-free equilibrium $E_u=(2.6\times10^9,0,0,0)$ at around $t=400$. Moreover, the dynamics of HBV DNA-containing capsids and virions is something different. Initially, both concentrations of capsids and virions increase progressively, reach the peak level, then decrease continuously and approach zero. Therefore, the disease-free equilibrium is globally asymptotically stable, and this supports the theoretical result stated in Theorem \ref{Theorem 5}. Additionally, this results also demonstrates the clearance of HBV infection in the presence of the effects of capsid recycling.
	\begin{center}
		\begin{figure}[h!]
			\centering
			\includegraphics[width=12cm,height=9cm]{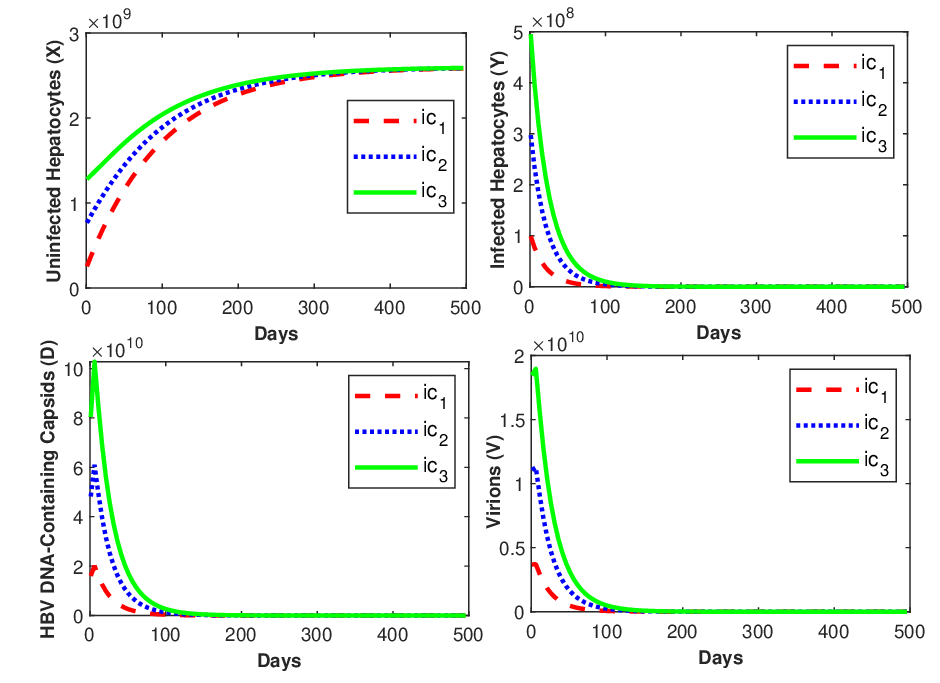}
			\caption{The dynamics of the system \eqref{eq1} while $R_0<1$ with three different initial states $ic_1, ic_2, ic_3$.}
			\label{fig_1}
		\end{figure}
	\end{center}
	\subsection{Endemic equilibrium ($R_0>1$)}
	In order to analyze the stability of endemic equilibrium numerically, we choose $k=3\times10^{-12}$. The dynamics in this case are shown in Figure \ref{fig_2}. It is seen that the concentration level of uninfected hepatocytes, infected hepatocytes, HBV DNA-containing capsids, and viruses initially oscillate for some time and slowly converge to the endemic equilibrium point asymptotically. Consequently, the numerical results agree with the theoretical results stated in Theorem \ref{Theorem 6}.
	\begin{figure}[h!]
		\centering
		\includegraphics[width=12cm,height=9cm]{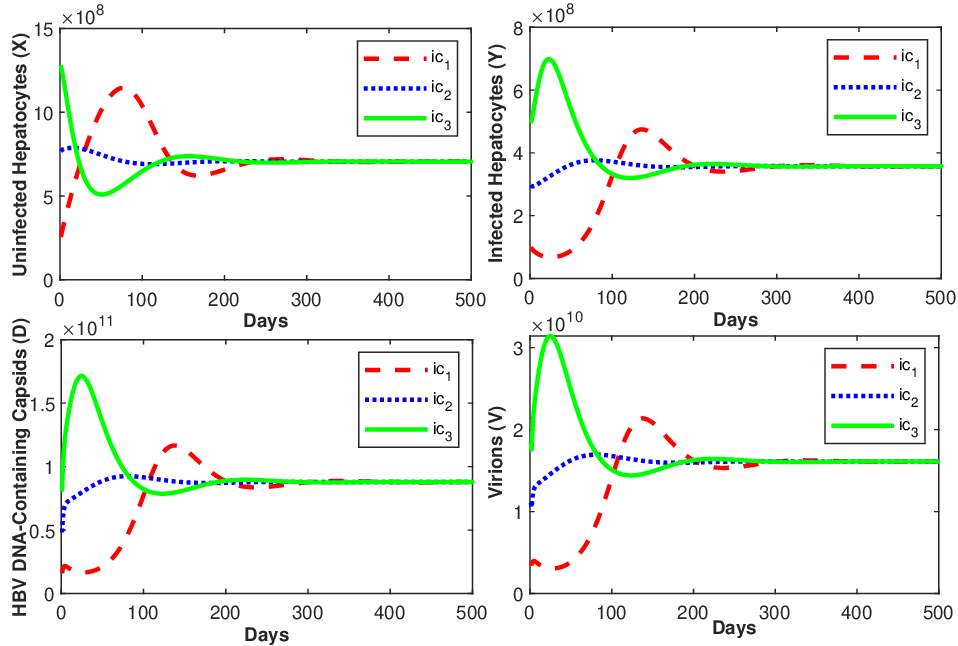}
		\caption{The dynamics of the system \eqref{eq1} while $R_0>1$ with three different initial states $ic_1, ic_2, ic_3$.}
		\label{fig_2}
	\end{figure}
	\section{Comparison of the model} \label{Comparison of the model}
	In this section, it is demonstrated how the dynamics of the system \eqref{eq1} changes when the capsid-to-capsid production rate ($\gamma$) is included in the model.
	\begin{enumerate}[(i)]
		\item \textbf{BMwoR:} Basic HBV model where the effects of ``recycling" of HBV DNA-containing capsid is not considered. In this case  $\gamma=0$ and $0<\alpha<1$.
		\item \textbf{BMwR:} Basic HBV model where the effects of ``recycling" of HBV DNA-containing capsid is considered. In this case  $0<\alpha<1,$ and $\gamma>0$.
	\end{enumerate}
	The variations in the concentration level of uninfected hepatocytes, infected hepatocytes, HBV capsids, and viruses are presented in Figure \ref{fig_3}. It is observed that when the effects of recycling of capsids is incorporated in the model, the equilibria remains stable, but the stability level of uninfected hepatocytes decreases significantly, whereas the stability levels of infected hepatocytes, HBV DNA-containing capsids, and virions increase in large amount. So, Figure \ref{fig_3} points out that the inclusion of recycling of capsids in the HBV  model makes momentous differences in the dynamics of infection.
	\begin{figure}[h!]
		\centering 
		\includegraphics[width=12cm,height=9cm]{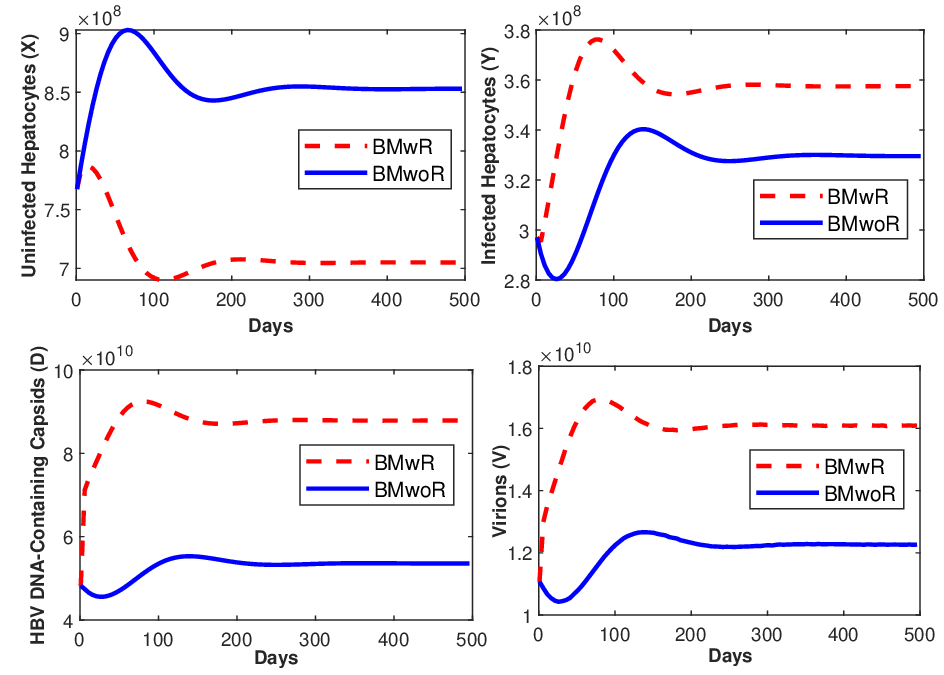}
		\caption{Comparison of model. The parameter values are taken from Table  \ref{Table-2}. Here $(i)$ BMwoR: Basic model without recycling of capsid $(ii)$ BMwR: Basic model with recycling of capsids.}
		\label{fig_3}
	\end{figure}
	\subsection{Effects of volume fraction ($\alpha$) of  HBV DNA-containing capsids}
	The effects of volume fraction of capsids ($\alpha$) on the system \eqref{eq1} are studied for two cases. In the first case,  impacts of recycling of capsids on the system \eqref{eq1} are ignored, while in the second case, it is considered. In Figure \ref{fig_4},  effects of $\alpha$  in the absence of recycling of capsids are demonstrated. The numerical simulation is performed for six values of $\alpha$, namely, $\alpha=0.5, 0.6, 0.7, 0.8, 0.9 ~\& ~1.0$ while keeping fixed  the other parameters. 
	The condition $R_s>0$ is satisfied for every value of $\alpha$. In all cases, $R_0>1$ \textit{i.e.} the system converges to corresponding endemic equilibrium points. Figure \ref{fig_4} shows that the uninfected hepatocytes and HBV capsids decline, but infected hepatocytes and viruses progressively increase while  $\alpha$ increases.  
	
	However, when the recycling effects of capsids are included in the model, the reverse results are seen for uninfected, infected hepatocytes, and viruses in Figure \ref{fig_5}. There is no change observed in the dynamics of capsids, but the stability level significantly increases for the same value of $\alpha$. Consequently, when $\alpha=1.0$ (in this case recycling of capsids is ignored), it is clear that the concentration level of uninfected hepatocytes is at a highest level while that of infected hepatocytes, HBV capsids, and viruses get stabilized at a lowest level. Thus, the results for these two cases are significantly different.		
	The low value of volume fraction of capsids ($\alpha$) implies that a less number of capsids can produce  new viruses, and a large number of capsids get accumulated inside the hepatocytes. Therefore, an accumulation of core particles (capsids) within the infected hepatocytes can be a cause of severe infection  rather than the rapidly release of viruses. Thus, the role of $\alpha$ cannot be ignored in the design of strategies for controlling HBV infection.

	\begin{figure}[h!]
		\centering
		\includegraphics[width=12cm,height=9cm]{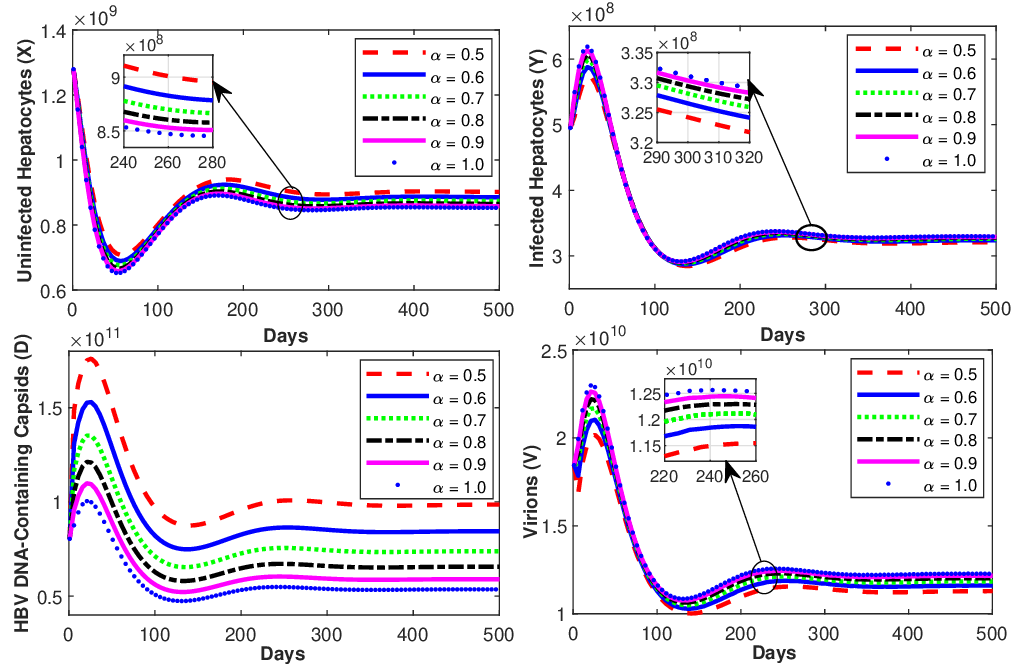}
		\caption{ The effects of $\alpha$ on the dynamical system \eqref{eq1} in absence of recycling of capsids ($\gamma=0$), and $ic_3$ is taken as the initial condition.}
		\label{fig_4}
	\end{figure}
	\begin{figure}[h!]
		\centering
		\includegraphics[width=12cm,height=9cm]{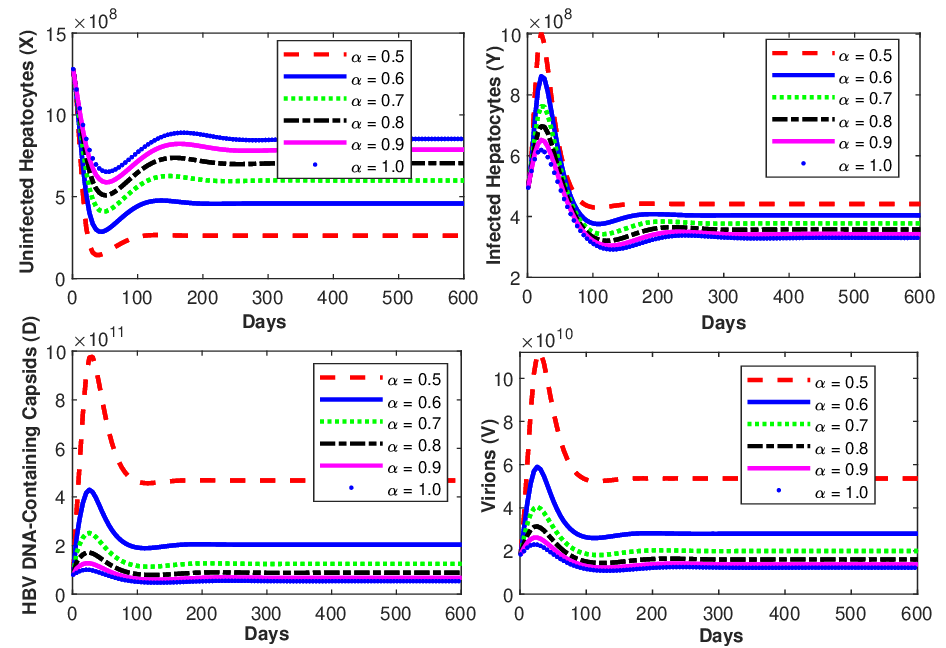}
		\caption{The effects of $\alpha$ on the dynamical system \eqref{eq1} when the effects of recycling of capsids  is considered ($\gamma>0$), and $ic_3$ is taken as the initial condition.}
		\label{fig_5}
	\end{figure}
	\subsection{Effects of capsid to capsid production rate ($\gamma$)}
	In Figure \ref{fig_6}, we present the impact of recycling rate or capsid-to-capsid production rate  ($\gamma$)  on all the four populations in the system \eqref{eq1}. The parameter values are same as in Table \ref{Table-2} except $\gamma$. Six  different values of $\gamma~ (=0.5, 1.0, 1.5, 2.0, 2.5, 3.0)$  are chosen  for simulation in such a way that the criteria for boundedness of solution $R_s>0$ is satisfied.  For these value of parameters, the corresponding values of $R_0$ are greater than unity. 
	The value of $R_0$  increases with increase of $\gamma$. Also, one more observation is seen that when $\gamma$ takes its minimum value zero {\textit{i.e.}} no capsid-to-capsid production rate is considered, system achieves its stability asymptotically with maximum concentration of uninfected hepatocytes and this concentration drops when $\gamma$ rises. It may be noted that the concentration level of infected hepatocytes decreases with the increase in $\gamma$ and attains its minimum value when $\gamma=0$. A similar scenario is observed for the HBV DNA-containing capsids and virions. In other words, the peak level for the virus compartment becomes smaller as $\gamma$ decreases. Therefore, the severity of the infection becomes less and disappears faster. Biologically, one can see that for high value of $\gamma$, the situation becomes more critical for the patient, and it is difficult to be cured.  So, this study indicates that the inclusion of capsid-to-capsid production rate in the model is crucial to study the dynamics of HBV infection in a more realistic way. 
	\begin{figure}[h!]
		\centering
		\includegraphics[width=12cm,height=9cm]{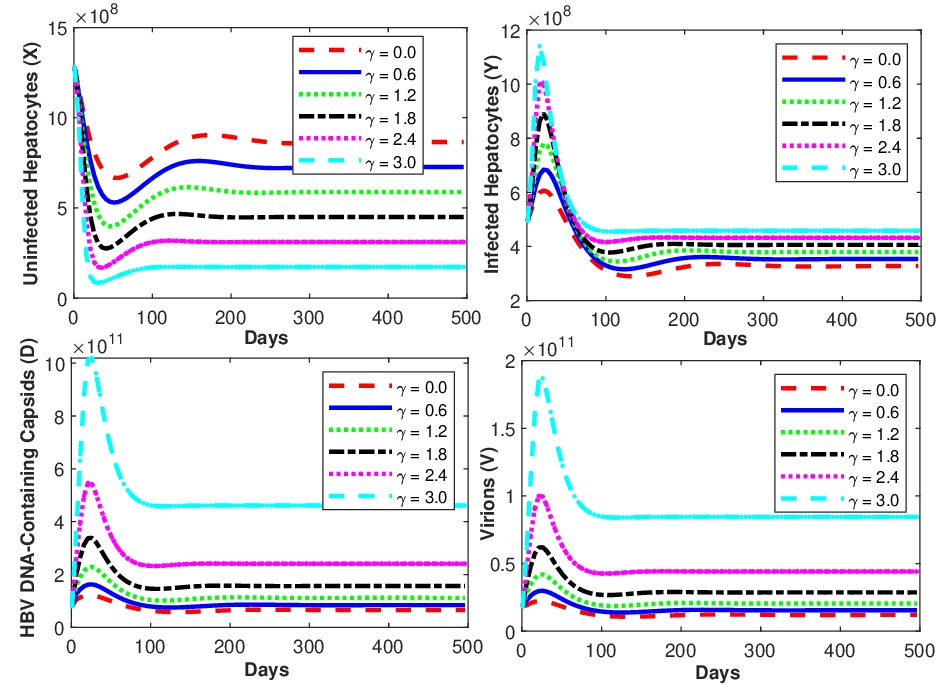}
		\caption{The effects of $\gamma$ on the dynamical system \eqref{eq1} w.r.t the initial condition $ic_3$.}
		\label{fig_6}
	\end{figure}
	\subsection{Effects of capsids to virus production rate ($\beta$)}
	For different values of $\beta$, the changes in dynamics of uninfected hepatocytes, infected hepatocytes, capsids, and viruses are shown in Figure \ref{fig_7}. In this case, the effects of recycling of capsids are not taken into account. Six values of $\beta$ (=0.6, 0.7, 0.8, 0.9, 1.0, 1.1) are considered keeping other parameters fixed  as shown in Table \ref{Table-2}. For each value of $\beta$, $R_s>0$ holds. When $\beta$ decreases, the concentrations of uninfected hepatocytes and capsids increase, while opposite trends are observed for infected hepatocytes and viruses. For a low value of $\beta$, the amount of viruses (newly produced from the infected hepatocytes) also remains low. 
	
	On the other hand, the effects of recycling of capsids in the  HBV dynamics reverse the impacts of $\beta$ which is shown in Figure \ref{fig_8}.  If $\beta$ decreases, concentration of  uninfected hepatocytes increases whereas the numbers of infected hepatocytes, HBV DNA-containing capsids, and viruses increase. For smaller values of $\beta$, the concentration of
	cccDNA increases inside the nucleus due the recycling of capsids. Therefore, the number of released virions into the extracellular space increases
	gradually over time. In a word,  recycling of capsids acts as a positive feedback
	loop in viral replication.
	In this case, the results indicate that rather than the rapid release of HBV viruses from the infected cell, the accumulation of HBV DNA-containing capsids inside infected cells can play a major role in the exacerbation of infection.
	A low value of $\beta$ makes things worse for the sufferer, thus making it difficult to cure. Small virion release rates ($\beta$) in the HBV replication process may be a risk factor for chronic hepatitis exacerbation over time. So, this discussion underlines the importance of recycling of capsids in cases of HBV infection.
	
	\begin{figure}[h!]
		\centering
		\includegraphics[width=12cm,height=10cm]{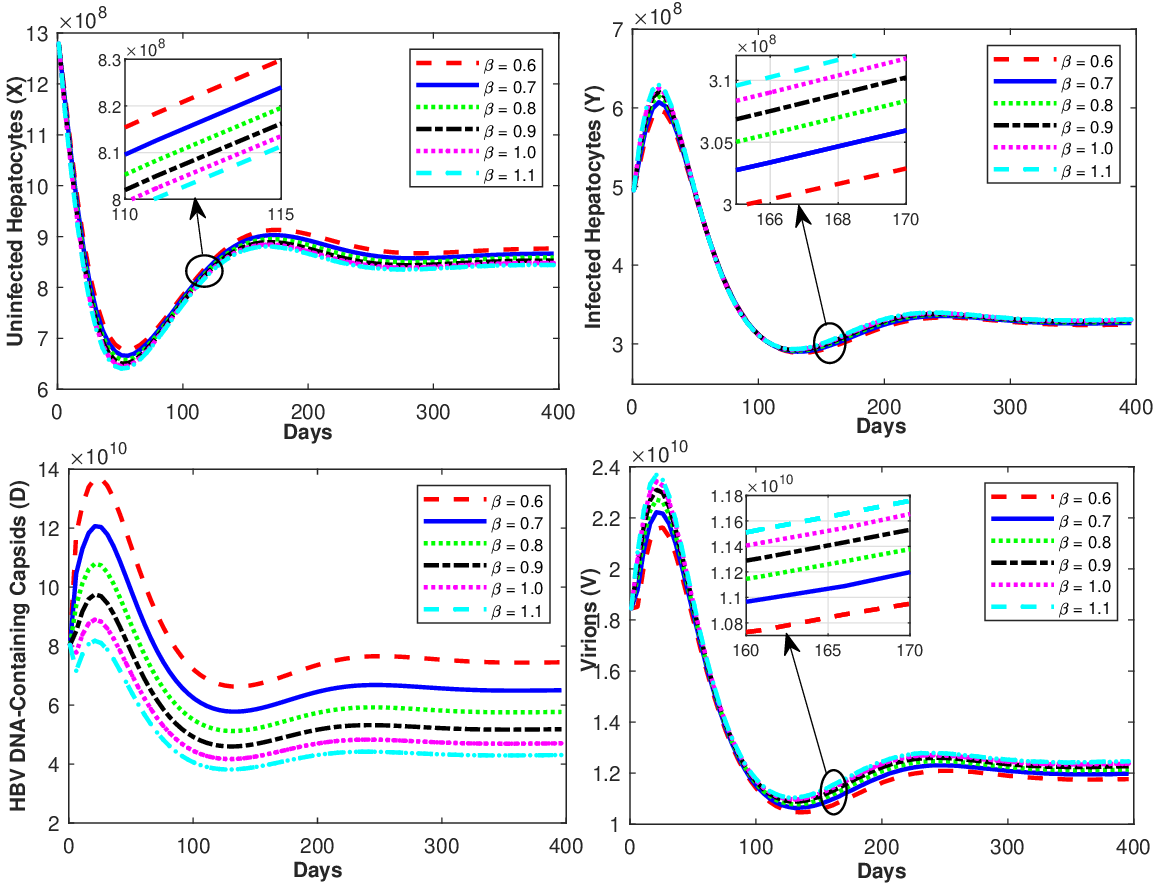}
		\caption{The effects of $\beta$ in the absence of recycling of capsids ($\gamma=0$) on the dynamical system \eqref{eq1} w.r.t the initial condition $ic_3$.}
		\label{fig_7}
	\end{figure}
	\begin{figure}[h!]
		\centering
		\includegraphics[width=12cm,height=9cm]{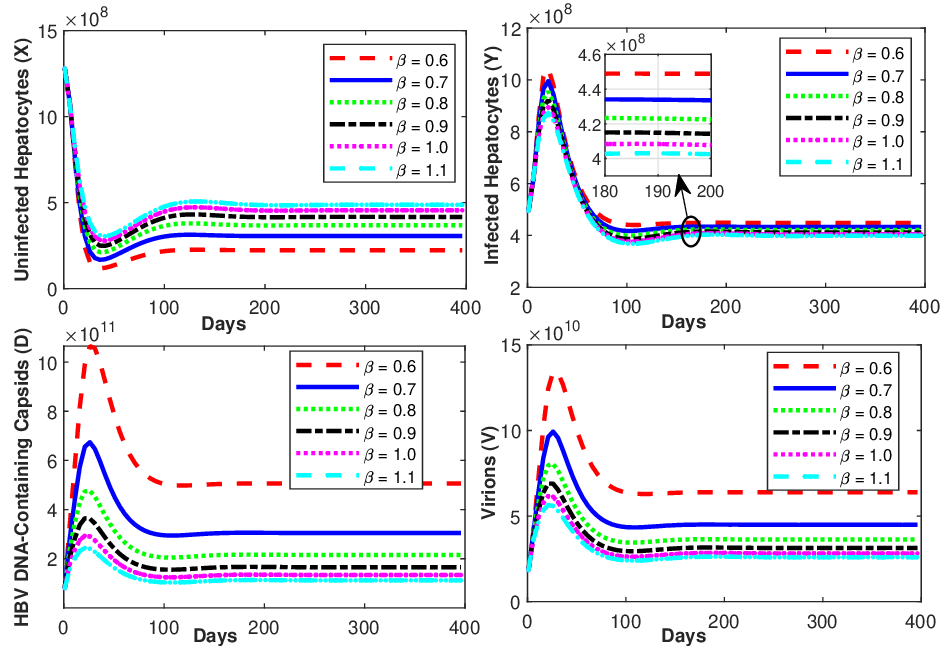}
		\caption{The effects of $\beta$ on the dynamical system \eqref{eq1} when effects of recycling of capsids is considered  ($\gamma>0$) w.r.t the initial condition $ic_3$.}
		\label{fig_8}
	\end{figure}
	In Table \ref{Table-Parameter effect}, all the above outcomes  that occurred as a result of parameter variation are listed.
	\begin{table}[h!]
		\footnotesize
		\begin{center}
			\begin{tabular}{|c|c|c|c|c|c|}
				\hline
				Recycling &Effect&Uninfected  & Infected 		& HBV				 & Hepatitis B \\
				Effect&of&Hepatocytes & Hepatocytes 	& DNA-containing 	 &  Viruses      \\
				($\gamma$)&parameter&($X$)&     ($Y$)       	& Capsids ($D$) 	 		 & $(V)$   		   \\
				\hline
				Not Considered& $\alpha$ increases & Decrease & Increase & Decrease & Increase\\
				\hline
				Considered& $\alpha$ increases & Increase & Decrease & Decrease & Decrease\\
				\hline
				Considered& $\gamma$ increases & Decrease & Increase & Increase & Increase\\
				\hline
				Not Considered& $\beta$ increases & Decrease & Increase & Decrease & Increase\\
				\hline
				Considered& $\beta$ increases & Increase & Decrease & Decrease & Decrease\\
				\hline
			\end{tabular}
			\vspace{0.1cm}
			
			\caption{\label{Table-1} 
				All the above outcomes occurred as a result of parameter variation.}
			\label{Table-Parameter effect}
		\end{center}
	\end{table}
	
	\begin{remark}
		From the above discussion, it is observed that the role of volume fraction of capsids ($\alpha$)  and virus production rate ($\beta$) in infection dynamics are similar in nature.
	\end{remark}
	\subsection{Elasticities of basic reproduction number with respect to the parameter $\alpha$, $\gamma$ and $\beta$}
	\noindent	
	The static quantity $R_0$ depends on the all parameters of the model \eqref{eq1}. In the prediction of evolution of HBV, $R_0$ plays an important role. The sensitivity analysis of $R_0$ is performed here in order to determine how $R_0$ responds to changes in parameters.  The elasticity of quantity $\mathbb{Q}$ with respect to the parameter $p$ is denoted by $\mathcal{E}_p^\mathbb{Q}$ \cite{2015_martcheva_introduction} and given by
	\begin{align*}
		\mathcal{E}_p^\mathbb{Q}=\frac{p}{\mathbb{Q}} \frac{\partial \mathbb{Q}}{\partial p}=\frac{\partial~ \ln \mathbb{Q}}{\partial ~\ln p}.
	\end{align*} 
	\noindent
	In general, elasticity of $\mathbb{Q}$ is positive if it increases with respect to $p$, and negative if it decreases with respect to $p$.
	\begin{align*}
		\mbox{Elasticity of}~R_0~\mbox{w.r.t}~\alpha~\left(\mathcal{E}_\alpha^{{R}_0}\right)&=\frac{\alpha}{R_0} \frac{\partial R_0}{\partial \alpha}=\frac{\delta -\gamma }{\alpha  (\beta +\gamma )-\gamma +\delta },\\
		\mbox{Elasticity of}~R_0~\mbox{w.r.t}~\gamma~\left(\mathcal{E}_\gamma^{{R}_0}\right)&=\frac{\gamma}{R_0} \frac{\partial R_0}{\partial \gamma}=\frac{(\alpha -1) \gamma +\delta }{\alpha  (\beta +\gamma )-\gamma +\delta },\\	
		\mbox{Elasticity of}~R_0~\mbox{w.r.t}~\beta~\left(\mathcal{E}_\beta^{{R}_0}\right)&=\frac{\beta}{R_0} \frac{\partial R_0}{\partial \beta}=-\frac{\gamma  (\alpha  c \delta -c \delta )}{\alpha  \beta  c \delta +\alpha  c \gamma  \delta -c \gamma  \delta +c \delta ^2}.
	\end{align*} 	
	\noindent
	One can obtain the followings by using the values of parameters from Table \ref{Table-2}: 	
	\begin{align*}
		\mathcal{E}_\alpha^{{R}_0}\approx-1.05, ~\mathcal{E}_\gamma^{{R}_0}\approx 0.23 ~\mbox{and}~ \mathcal{E}_\beta^{{R}_0}\approx-0.14.
	\end{align*}	
	\noindent
	This shows  that 1\% increase in $\alpha$, $\beta$ and $\gamma$ will produce 1.05\%, 0.14\%   decrease and 0.23\% increase
	in $R_0$. The sensitivity analysis reveals that $\alpha$ has greater positive impacts on $R_0$ in magnitude. So, $\alpha$ can be chosen as a disease controlling parameter.
	\begin{figure}[h!]
		\begin{center}
			\includegraphics[width=12cm,height=8cm]{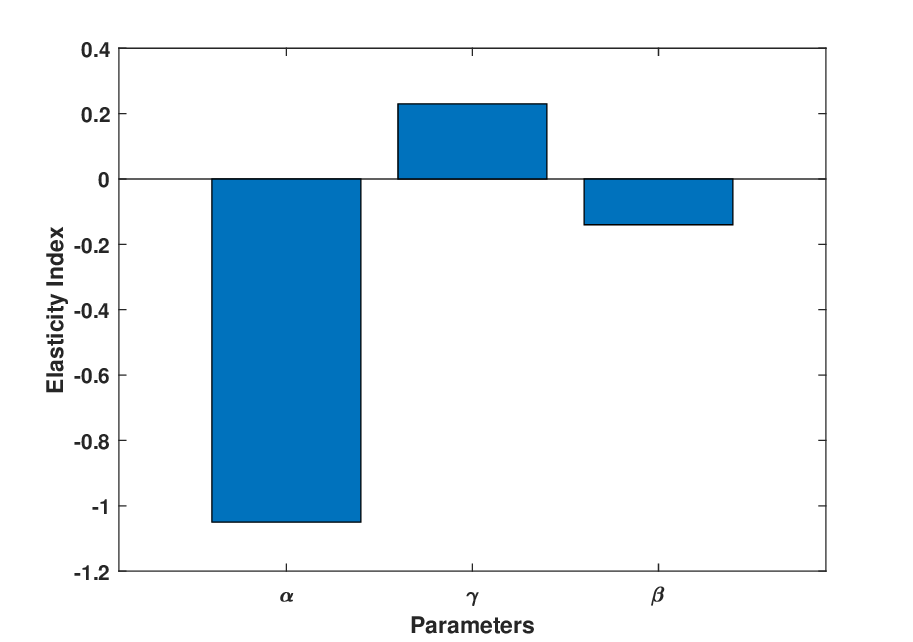};
			\caption{ Elasticities of basic reproduction number with respect to parameters $\alpha$, $\gamma$ and $\beta$.}
			\label{fig_10}
		\end{center}
	\end{figure}
	\section{Global sensitivity analysis of the model} \label{GSA}
	The accuracy of results of a mathematical model related to some biological phenomena often becomes poor  because of uncertainties in experimental data which are utilized in the estimation of model parameters. Recently, many authors study the effects of single parameter keeping all others parameters fixed at their estimated values. This type of sensitivity analysis is called  local sensitivity analysis. But local sensitivity analysis doesn't provide  the proper information of uncertainty and sensitivity of the parameters.  In order to find out the contributions of each model parameters universally in HBV infection dynamics, the global sensitivity analysis is performed using the technique  ``Latin hypercube sampling-partial rank correlation coefficient'' (LHS-PRCC) described by Marino et al. \cite{2008_marino_methodology}. 
	\subsection{Latin hypercube sampling (LHS)-Partial rank correlation coefficient (PRCC)}
	Latin hypercube sampling is one kind of Monte  Carlo class of sampling methods. In 1979, McKay et al. \cite{1979_Mckay_comparison}  first introduced this sampling method. With the help of LHS, sample inputs of the model are arranged within a "hypercube of dimensions $p$", where $p$ represents the number of model parameters. For our proposed model \eqref{eq1}, the number of model parameters $(p)$ is equal to 9. A probability density function (pdf) is employed for sampling parameter values based on parameter ranges partitioned into intervals. In this study, the uniform distribution is chosen for all parameters depending on  a priori information and existing data. The model is then simulated iteratively over all $p$-tuples parameter pairs. It is recommended that the sample size $N$ be at least $(p+1)$, but it is better to take a larger sample size to ensure the desired accuracy of the results . Here, the sample size  is set to 1000.
	
	The correlation coefficient (CC) measures the strength of a linear relationship between the inputs and the outputs. The  CC  is calculated between the input variable ($X$) and output variable ($Y$) as follows:
	\begin{align*}
		r=\dfrac{\sum(X-\bar{X})(Y-\bar{Y})}{\sqrt{\sum(X-\bar{X})^2\sum(Y-\bar{Y})^2}},
	\end{align*}
	where $\bar{X}$ and $\bar{Y}$ represent the  sample means of $X$ and $Y$ respectively and $r\in[-1, 1]$. In case of raw data of  $X$ and  $Y$, the coefficient $r$ is known as sample or Pearson correlation coefficient. The CC ($r$) is called Spearman or  rank correlation coefficient if the data are rank-transformed.  By using LHS-PRCC, one can derive insightful conclusions about how the model parameters influence on the outputs of a system. There are several publications that describe the PRCC method in detail\cite{2008_marino_methodology,2021_afsar}.
	\subsection{Scatter plots: The monotonic relationship between input and output variables}
	Besides the improvement and  generalization of a dynamical system, it has attracted the attention of many researchers to know how the outputs are affected if the  parameters' values  vary in a reasonable range. In the practical field of application especially in virus dynamics model, it is very important and essential to study the sensitivity of parameters. In such cases, PRCC values can provide useful information. PRCC  also can assist us to identify which set of parameters is the most significant for achieving some specific goals such as control or regulatory mechanisms, reduce viral load, increase immune response, proposing any new therapy and optimization of drug usage. PRCC is also capable of identifying both positive and negative correlations on model outputs.  In order to analyze the sensitivity of parameters, the baseline values are taken from  Table \ref{Table-2}. To comprehensively analyze the system globally, we opted to vary all parameters from 80\% to 120\% of their base values. Simulation results of the proposed  model \eqref{eq1} are visualized by scatter plots on Figure \ref{Scatter_uninfected} - Figure \ref{Scatter_infected}. The scatter plot for the capsids class is not displayed in this analysis since no substantial distinction is observed when comparing it to the scatter plot of viruses.   The PRCC values of all  parameters are calculated at day 300 with respect to the dependent variables. The positive correlation of a compartment to a model parameter (PRCC value positive) ensures that if the value of this parameter increases individually or simultaneously, the concentration of the compartment increases accordingly. On the other hand, negative correlation (PRCC value negative) tells us the opposite aspects.

	Based on the PRCC values, the model parameters are arranged in descending order for the uninfected, infected, and virus classes as follows: 
	\begin{itemize}
		\item \textbf{Uninfected hepatocyte}:  $\lambda,~\alpha,~c,~\beta,~\delta,~a,~k,~\gamma,~\mu$.
		\item \textbf{Infected hepatocyte}: $k,~a,~\lambda,~~\gamma,~\beta,~\mu,~\delta,~c,~\alpha$.
		\item  \textbf{Virus}: $a,~k,~\lambda,~\gamma,~\beta,~\mu,~\delta,~\alpha,~c$.
	\end{itemize}

	
	Global sensitivity analysis uncovers numerous new and remarkable findings, which are outlined as follows:
	\begin{enumerate}
		\item The parameters $\lambda,~k,~a$ are identified as the most positively sensitive parameters for uninfected hepatocytes, infected hepatocytes, and the viruses. On the other hand, the parameters $\mu,~\alpha,~c$ are found to be the most negatively sensitive parameters for the same entities. 
		\item In chronic infection, it has been observed that the virus production rate ($\beta$) has a relatively less influence on the overall infection. However, the infection rate ($k$) itself plays a crucial role in driving and sustaining chronic infection.
		\item The recycling rate ($\gamma$) is the second most negatively sensitive parameter for uninfected liver cells. Therefore, the recycling of capsids can indeed act as a positive feedback loop in the context of infection.
		\item The volume fraction of capsids is identified as the most negatively sensitive parameter for the infected compartment. That means if less number of  capsids involved in producing new virions and a larger number of capsids undergo in recycling and as a result this would make the infection more severe.
	\end{enumerate}
	
	\begin{figure}[h!]
		\begin{center}
			\includegraphics[width=16cm,height=9cm]{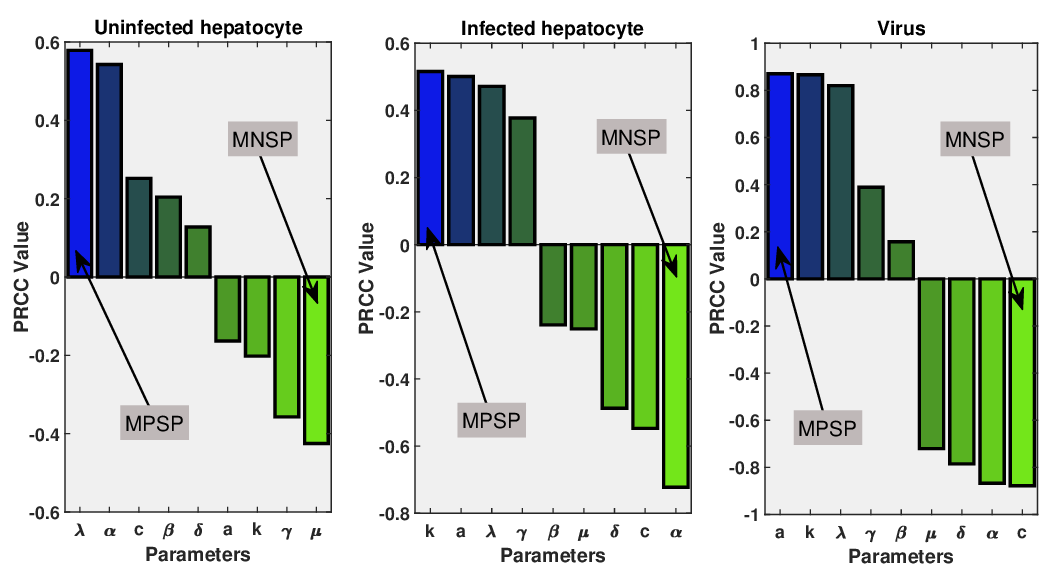};
			\caption{PRCC values of parameters corresponding to uninfected, infected hepatocyte and virus classes are plotted. Here, MNSP: Most Negatively Sensitive Parameter;  MPSP: Most Positively Sensitive Parameter.}
			\label{PRCC_Parameters}
		\end{center}
	\end{figure}
	\begin{figure}[h!]
		\begin{center}
			\includegraphics[width=15cm,height=10cm]{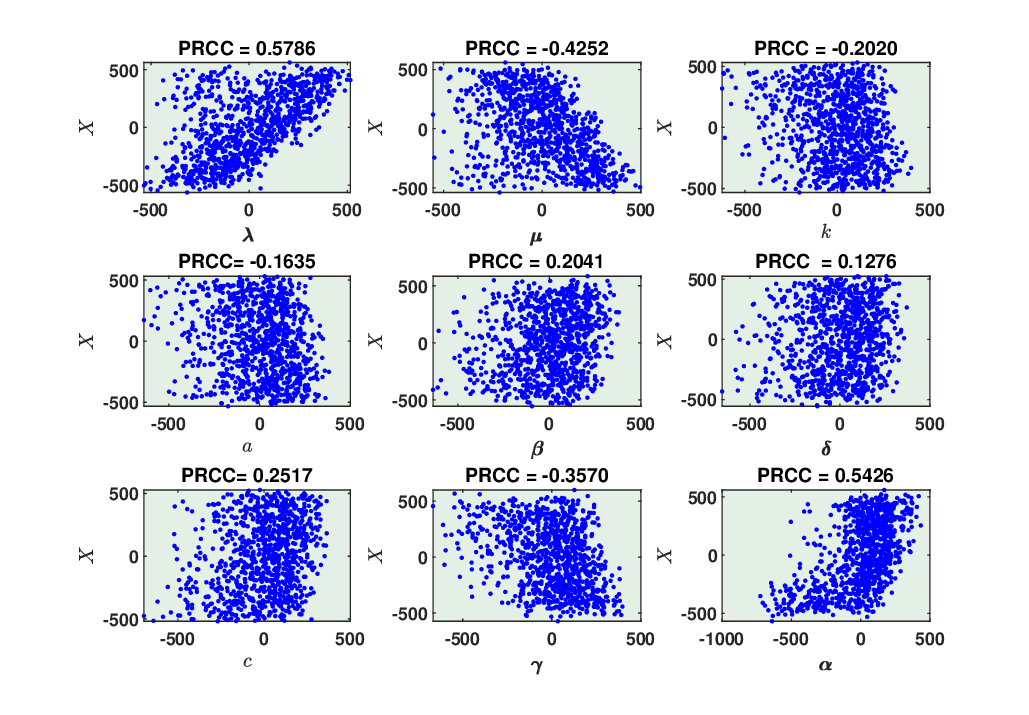};
			\caption{Scatter plot for uninfected hepatocytes ($X$). PRCC value of each parameter are shown on the title on each subplot.}
			\label{Scatter_uninfected}
		\end{center}
	\end{figure}
	\begin{figure}[h!]
		\begin{center}
			\includegraphics[width=15cm,height=10cm]{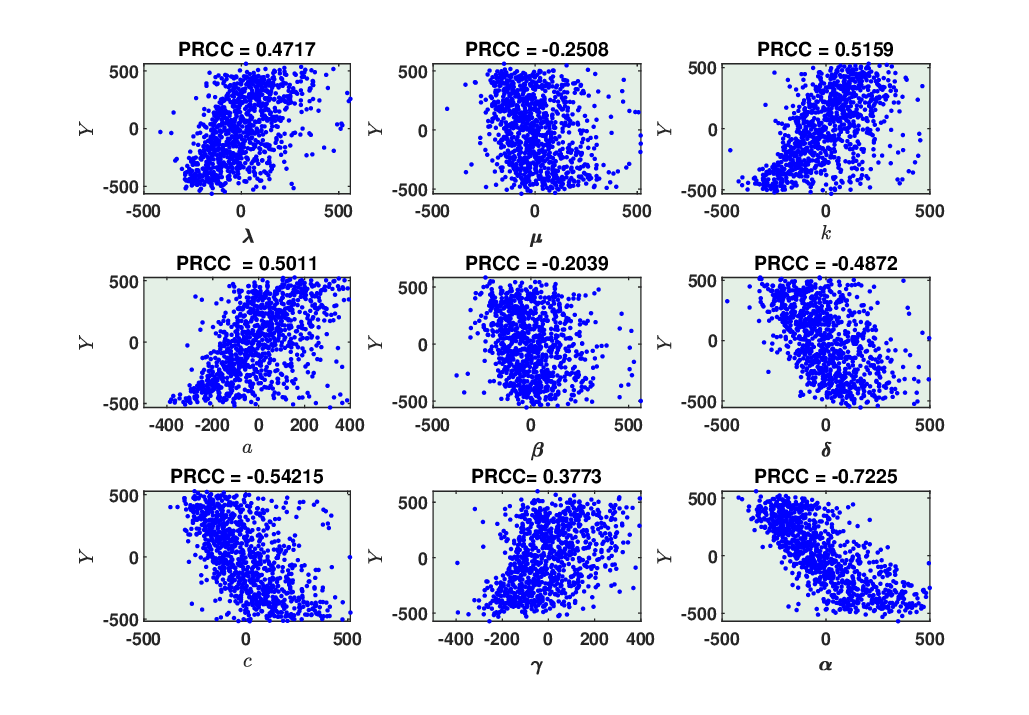};
			\caption{Scatter plot for infected hepatocytes ($Y$). PRCC value of each parameter are shown on the title on each subplot.}
			\label{Scatter_infected}
		\end{center}
	\end{figure}
	\begin{figure}[h!]
		\begin{center}
			\includegraphics[width=15cm,height=10cm]{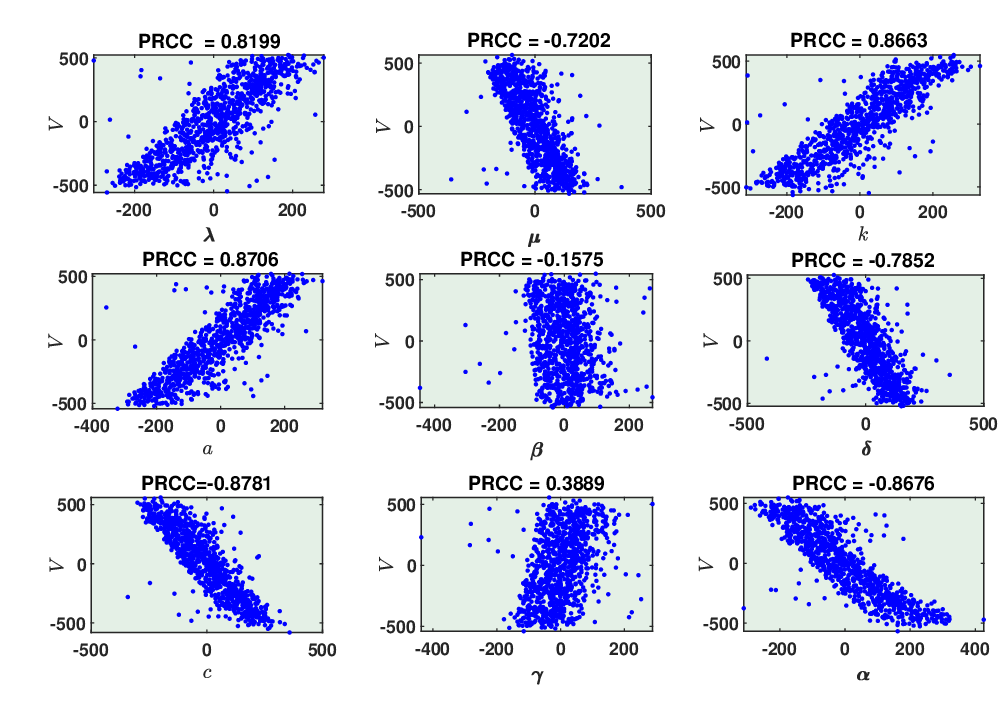};
			\caption{Scatter plot for Virus ($V$). PRCC value of each parameter are shown on the title on each subplot.}
			\label{Scatter_virus}
		\end{center}
	\end{figure}
	\clearpage
	\section{Conclusions} \label{Conclusions}
	In this present study,  hepatitis B virus infection dynamics is modeled  based on the biological findings.  In order to describe this viral infection in a more realistic way, the recycling effects of capsids are incorporated in this model. By including the recycling effects, we have noticed a paradigm shift in the outcomes of the proposed model. The non-negativity and boundedness of the solutions establish the feasibility of the system. The stability analysis of the system indicates that both the equilibrium points of the proposed model are globally asymptotically stable under some conditions \textit{i.e.} the patient will either achieve a full recovery, or the infection will persist for the rest of the life.
	Upon comparing the model solution with experimental data collected from four chimpanzees, it is seen that the model solution agrees well with the
	experimental data. Hence, the proposed model effectively captures and represents the intricate dynamics of HBV infection, making it a more realistic and reliable tool for studying this disease. Comparing with some  other relevant studies in the literature, it is also concluded that this model describes richer dynamics behavior of the infection. In addition, it is further observed that due to recycling the viral load increases considerably.	
	\\
	\noindent
	From the simulated results, the following findings are observed.
	\begin{enumerate}[(i)]

		\item Most of the mathematical  models on HBV infection developed so far, underestimate the production of virions and suppress the production of capsids as the recycling effects of capsids are ignored. Consequently, these models fail to capture the actual dynamics of HBV infection,  whereas, the proposed model shows a more realistic production of virions and the actual dynamics of the infection.
		
		\item 
		Recycling rate of capsids ($\gamma$) is one of the deciding parameters in the model to determine the severity of infection. 
		So, it is very important to pay attention to  this kind of parameter while proposing any control strategy for this disease. 
		
		\item Probably, this study analyses the effects of volume fraction of capsids ($\alpha$) on disease dynamics for the first time.	
		It is found that the inclusion of recycling of capsids reverses the effects of volume fraction on  the infection. This is a striking outcome of the present study that changes the usual understandings about viral dynamics. 
		So, volume fraction of capsids emerges as a viable candidate for a disease-controlling parameter.
		
		
		\item It is also observed that the number of released viruses  increases in spite of low virus production rate due to recycling of capsids. Though this result appears to be contradictory to the known fact, but our study has clearly explained this new findings. In order to gain deeper insights about this infection, the emergence of this unusual behavior becomes very important. On the other word, the recycling of capsids acts as a positive feedback loop in this viral infection.
		\item Based on the values of Partial rank correlation coefficients, the global sensitivity analysis unequivocally identifies that the disease progression is highly influenced by the volume fraction of capsids as well as recycling rate. These findings highlight the pivotal role of these factors in shaping the dynamics of the disease and warrant further attention in future studies.
		\item The strong concurrence between the model solution and the experimental data  substantiates that the proposed model is more realistic and reliable.
	\end{enumerate}
	As mentioned above, our model provides a theoretical backbone of the mechanism causing the exacerbation during the
	chronic HBV infection. This is  a new, and  relatively simple mathematical model that can  describe the infection dynamics more accurately. Using these new findings, this model can be applied to a variety of clinical trials and for the formulation of new drugs.
	\section*{Declarations}
	\subsection*{Ethics approval and consent to participate}
	\noindent
	Not applicable.
	\subsection*{Consent for publication}
	\noindent
	Not applicable.
	\subsection*{Availability of data and materials}
	\noindent
	 Available once the manuscript accepted.
	 \subsection*{Competing interests}
	 \noindent
	 The authors declare that they have no competing interests.
	 \subsection*{Funding}
	 \noindent
	 Council of Scientific \& Industrial Research-SRF Fellowship scheme
	 (File No: 09/731(0171)/2019-EMR-I).
	\subsection*{Acknowledgments}
	\noindent
	The first author also thanks the research facilities received from the Department of Mathematics, Indian Institute of Technology Guwahati, India.
	\subsection*{Authors' information}
	\noindent
	Rupchand Sutradhar, Department of
	Mathematics, Indian Institute of
	Technology Guwahati, Guwahati, Assam
	781039, India.
	Email: rsutradhar@iitg.ac.in\\
	D C Dalal, Department of
	Mathematics, Indian Institute of
	Technology Guwahati, Guwahati, Assam
	781039, India.
	Email: durga@iitg.ac.in

	%
	
%

	%
	%
	

	\section*{Appendix A ~~~~Proof of Non-negativity of the solutions}
\begin{proof}
	It is clear that $\left[\frac{dX}{dt}\right]_{X(t)=0}=\lambda$, $\left[\frac{dY}{dt}\right]_{Y(t)=0}=kVX$, $\left[\frac{dD}{dt}\right]_{D(t)=0}=aY$, $\left[\frac{dV}{dt}\right]_{V(t)=0}=\alpha\beta D$. Since, $\left[\frac{dX}{dt}\right]_{X(t)=0}> 0$, so $X(t)$ is always non-negative $\forall$ $t>0$.
	
	In order to prove the non-negativity of $Y(t)$, it is essential to show first the non-negativity of  $V(t)$. With that aim, it is assumed that
	$V(t)$ does not satisfy the non-negativity condition \textit{i.e.} $\exists$ a time, say $t_V$ such that,	
	\begin{align*}
		t_V=\inf \left\{ t:~t>0,~V(t)=0,\frac{dV}{dt}\leq0\right\}.
	\end{align*}		
	\noindent		
	This implies that  $\left[\frac{dV}{dt}\right]_{V(t_V)=0}=\alpha\beta D(t_V)\leq0$ \textit{i.e.} $\exists$ a time $t_D$ such that		
	\begin{align*}
		t_D=\inf \left\{ t: ~t>0, ~D(t)=0, ~\dfrac{dD}{dt}\leq0 \right\},
	\end{align*}		
	\noindent	
	which gives  $\left[\frac{dD}{dt}\right]_{D(t_D)=0}=a Y(t_D)\leq0$. One may also find a time $t_Y$ such that	
	\begin{align*}
		t_Y=\inf \left\{ t:~t>0,~Y(t)=0,~\frac{dY}{dt}\leq0\right\}.
	\end{align*}	
	\noindent	
	Clearly, $t_V>t_D>t_Y$. It is clear that  $\left[\frac{dY}{dt}\right]_{Y(t_Y)=0}=k V(t_Y)X(t_Y)\leq 0$ $\implies$ $ V(t_Y)\leq0$. Since, $V(t_Y)>0$, we arrive at a contradiction to  the definition of $t_Y$. So,
	$V(t)\geq 0,~ \forall~ t>0$. Consequently, $D(t)\geq 0$ and $Y(t)\geq 0,~\forall ~t>0$.	
\end{proof}
	%
	%
	%
	%
	%
	%
	%
	%
	\clearpage

\begin{thebibliography}{10}
		\bibliographystyle{unsrt}
		\expandafter\ifx\csname url\endcsname\relax
		\def\url#1{\texttt{#1}}\fi
		\expandafter\ifx\csname urlprefix\endcsname\relax\def\urlprefix{URL }\fi
		\expandafter\ifx\csname href\endcsname\relax
		\def\href#1#2{#2} \def\path#1{#1}\fi
		%
		\bibitem{WHO_2021}
		Hepatitis B,
		\url{https://www.who.int/news-room/fact-sheets/detail/hepatitis-b}, 24 June 2022 .
		
		\bibitem{2001_Whalley}
		S.A. Whalley, J.M. Murray, D. Brown, G.J. Webster, V.C. Emery, 
		G.M. Dusheiko and  A.S. Perelson, Kinetics of acute hepatitis B virus infection in humans. J Exp Med 193(7) (2001) 847--854.
		
		\bibitem{2007_Ciupe_Role}
		S.M. Ciupe, R.M. Ribeiro, P.W. Nelson, G. Dusheiko and A.S. Perelson, The role of cells refractory to productive infection in acute hepatitis B viral dynamics. Proc Natl Acad Sci USA 104(12) (2007) 
		5050--5055.
		
		\bibitem{2002_Ribeiro}
		R.M. Ribeiro,  A. Lo and A.S. Perelson, Dynamics of hepatitis B virus infection, Microbes Infect 4(8) (2002) 829--835.
		
		
		
		%
		
		\bibitem{2002_lewin_hepatitis}
		S. Lewin, T. Walters and S. Locarnini, Hepatitis B treatment: rational combination chemotherapy based on viral kinetic and animal model studies. Antiviral Res 55(3) (2002) 381--396.
		
		\bibitem{2007_Guo_characterization}
		H. Guo, D. Jiang, T. Zhou, A. Cuconati, T.M.  Block and  J.T. Guo,
		Characterization of the intracellular deproteinized relaxed circular DNA of
		hepatitis B virus: an intermediate of covalently closed circular DNA
		formation. J Gen Virol 81(22) (2007) 12472--12484.
		
		\bibitem{2021_Fatehi}
		F. Fatehi, R.J. Bingham, E.C. Dykeman, N. Patel,P.G. Stockley and R. Twarock, An intracellular model of hepatitis B viral infection: An in silico platform for comparing therapeutic strategies. Viruses 13(1) (2021) 11.
		
		\bibitem{2005_Murray}
		J.M. Murray, S.F. Wieland, R.H. Purcell and F.V. Chisari, Dynamics of
		hepatitis B virus clearance in chimpanzees. Proc Natl Acad Sci USA 102(49) (2005) 17780--17785.
		
		\bibitem{2021_Saraceni_review}
		C. Saraceni and J. Birk, A review of hepatitis B virus and hepatitis C virus
		immunopathogenesis. J Clin Transl Hepatol 9(3) (2021) 409--418.
		
		\bibitem{2021_prifti}
		G.M. Prifti, D. Moianos, E. Giannakopoulou, V. Pardali, J.E. Tavis and
		G. Zoidis, Recent advances in hepatitis B treatment. Pharmaceuticals 14(5) (2021) 417.
		
		
		
		
		
		%
		\bibitem{2023_nayeem}
		    Nayeem J, Podde CN and Salek MA, SENSITIVITY ANALYSIS AND IMPACT OF AN IMPERFECT VACCINE OF TWO STRAINS OF HEPATITIS B VIRUS INFECTION, Journal of Biological Systems, pp 1-22. 
%
		\bibitem{2010_ji}
		Ji Y, Min L, Ye Y, Global analysis of a viral infection model with application to HBV infection. Journal of Biological Systems, 18(02), pp.325-337, 2010.
		
		
	   \bibitem{2019_hui}
		Hui H,  Nie LF, Analysis of a stochastic HBV infection model with nonlinear incidence rate. Journal of Biological Systems, 27(03), pp.399-421, 2019.
		
		
		
		%
		%
		\bibitem{2012_wang_mathematical}
		L. Wang and R. Xui, Mathematical analysis of an improved hepatitis B virus model,
		Int. J. Biomath. 5(1) (2012) 1250006.
		\bibitem{2015_Tchinda}
		Tchinda PM, Tewa JJ, Mewoli B, Bowong S, A theoretical assessment of the effects of distributed delay on the transmission dynamics of Hepatitis B. Journal of Biological Systems, 23(03), pp.423-455, 2015
		
		
		
		\bibitem{2015_Moualeu}
		Moualeu DP, Mbang J, Ndoundam R, Bowong S, Modeling and analysis of HIV and hepatitis C co-infections. Journal of Biological Systems, 19(04), pp.683-723, 2011.
		
		
		
		
		
		
		
		
		
		\bibitem{1996_Nowak}
		M.A. Nowak, S. Bonhoeffer, A.M. Hill, R. Boehme, H.C. Thomas, H. McDade, Viral dynamics in hepatitis B virus infection.  Proc Natl Acad Sci USA 93(9) (1996) 4398--4402.
		
		\bibitem{2003_Wodarz}
		D. Wodarz, Hepatitis C virus dynamics and pathology: the role of CTL and
		antibody responses. J Gen Virol  84(7) (2003) 1743--1750.		
		
		\bibitem{2010_Wang}
		K. Wang,  A. Fan and A. Torres, Global properties of an improved hepatitis B virus model. Nonlinear Anal Real World Appl 11(4) (2010) 3131--3138.
		
		
		%
		
		
		
		\bibitem{2009_Eikenberry}
		S. Eikenberry, S. Hews, J.D. Nagy and Y. Kuang, The dynamics of a delay model of HBV infection with logistic hepatocyte growth. Math Biosci Eng 6 (2009) 1--17.	
		
		%
		
		
		%
		
		
		
		
		
		
		%
		%
		
		
		
		
		
		
		
		
		
		
		\bibitem{2008_Min}
		L. Min, Y. Su and Y. Kuang, Mathematical analysis of a basic virus infection model with application to HBV infection. Rocky Mt J Math 38(5) (2008)   1573--1585.
		
		\bibitem{2014_Chen}
		X. Chen, L. Min, Y. Zheng, Y. Kuang and Y. Ye, Dynamics of acute hepatitis B virus infection in chimpanzees. Math Comput Simul 96 (2014) 157--170.
		
		\bibitem{2008_Gourley}
		S.A. Gourley, Y. Kuang and J.D. Nagy, Dynamics of a delay differential equation model of hepatitis B virus infection. J Biol Dyn 2(2) (2008)
		140--153.
		
		\bibitem{2010_Hews}
		S. Hews, S. Eikenberry,  J.D. Nagy and Y. Kuang, Rich dynamics of a hepatitis B viral infection model with logistic hepatocyte growth. J Math Biol 60(4) (2010) 573--590.
		
		\bibitem{2009_Huang}
		G. Huang, W. Ma and Y. Takeuchi, Global properties for virus dynamics model with beddington-deangelis functional response. Appl Math Lett 22(11) (2009)
		1690--1693.
		
		\bibitem{2010_YuJi}
		Y. Ji, L. Min, Y. Zheng and Y. Su, A viral infection model with periodic immune response and nonlinear CTL response. Math Comput Simul 80(12) (2010)
		2309--2316. 
		
		\bibitem{2018_fatehi_nkcell}
		F.F. Chenar, Y. Kyrychko and K. Blyuss, Mathematical model of immune response to hepatitis B. J Theor Biol 447 (2018) 98--110.
		
		\bibitem{2006_Murray}
		J.M. Murray, R.H. Purcell and S.F. Wieland, The half-life of hepatitis B
		virions. Hepatology 44(5) (2006) 1117--1121. 
		
		\bibitem{2009_Asabe}
		Asabe, Shinichi and Wieland, Stefan F and Chattopadhyay, Pratip K and Roederer, Mario and Engle, Ronald E and Purcell, Robert H and Chisari, Francis V, The size of the viral inoculum contributes to the outcome of hepatitis B virus infection.
		Journal of virology 83(19)(2009) 9652--9662.
		
		
		
		
		
		\bibitem{2015_Manna}
		K. Manna and S.P. Chakrabarty, Chronic hepatitis B infection and HBV
		DNA-containing capsids: Modeling and analysis. Commun Nonlinear Sci Numer Simul 22(1-3) (2015) 383--395.
		
		\bibitem{2018_Danane_optimal}
		J. Danane, A. Meskaf and K. Allali, Optimal control of a delayed hepatitis B viral infection model with HBV DNA-containing capsids and CTL immune response. Optim Control Appl Methods 39(3) (2018) 1262--1272.
		
		\bibitem{2018_Guo}
		T. Guo, H. Liu, C. Xu and F. Yan, Global stability of a diffusive and delayed HBV infection model with HBV DNA-containing capsids and general incidence rate. Discrete Contin Dyn Syst-B 23(10) (2018) 4223. 
		
		\bibitem{2021_liu_age}
		S. Liu and R. Zhang, On an age-structured hepatitis B virus infection model with HBV DNA-containing capsids. Bull Malaysian Math Sci Soc 44(3) (2021) 1345--1370.
		
		
		
		\bibitem{2021_diogo_review}
		J. D. Dias, N. Sarica and C. Neuveut, Early steps of hepatitis B life cycle: From capsid nuclear import to cccDNA formation. Viruses 13(5) (2021) 757.
		
		\bibitem{2018_Danane}
		J. Danane and K. Allali, Mathematical analysis and treatment for a delayed
		hepatitis B viral infection model with the adaptive immune response and
		DNA-containing capsids. High-throughput 7(4) (2018) 35.
		
		%
		
		\bibitem{1990_Diekmann}
		O. Diekmann, J.A.P Heesterbeek,  J.A. Metz, On the definition and the
		computation of the basic reproduction ratio $R_0$ in models for infectious
		diseases in heterogeneous populations. J Math Biol 28(4) (1990) 365--382.
		
		\bibitem{2002_Driessche_next_g_m}
		P. Van den Driessche and J. Watmough, Reproduction numbers and sub-threshold
		endemic equilibria for compartmental models of disease transmission. Math Biosci 180(1-2) (2002)  29--48.
		
		\bibitem{2005_Heffernan_next_g_m}
		J.M. Heffernan, R.J. Smith and L.M. Wahl, Perspectives on the basic
		reproductive ratio. J R Soc Interface. 2(4) (2005) 
		281--293.
		
		
		
		\bibitem{2015_martcheva_introduction}
		M. Martcheva, An Introduction to Mathematical Epidemiology, Vol.~61, Springer,
		(2015).
		
		\bibitem{2013_Perko_differential}
		L. Perko, Differential Equations and Dynamical Systems, Vol. 7, Springer Science \& Business Media, (2013).
		\bibitem{2012_Kajiwara}
		T. Kajiwara, T. Sasaki and Y. Takeuchi, Construction of lyapunov functionals for delay differential equations in virology and epidemiology. 	Nonlinear Anal Real World Appl 13(4) (2012)  1802--1826.
		
		\bibitem{2007_Dahari}
		H. Dahari, A. Lo, R.M. Ribeiro and A.S. Perelson, Modeling hepatitis C virus dynamics: Liver regeneration and critical drug efficacy. J Theor Biol 247(2) (2007) 371--380.
		.		
		\bibitem{2008_marino_methodology}
		Simeone Marino, Ian~B Hogue, Christian~J Ray, and Denise~E Kirschner. A methodology for performing global uncertainty and sensitivity analysis in systems biology. Journal of theoretical biology, 254(1) (2008) 178--196.
		\bibitem{2021_afsar}
		Md Afsar Ali, S.A. Means, Harvey Ho, Jane Heffernan,
		Global sensitivity analysis of a single-cell HBV model for viral dynamics in the liver. Infectious Disease Modelling, 6, (2021) 1220-1235.
		
		
		
		
		
		
		\bibitem{1979_Mckay_comparison}
		R.~J.~Beckman McKay, M.~D. and W.~J. Conover. A comparison of three methods for selecting values of input variables in the analysis of output from a computer code.
		Technometrics, 21(2) (1979) 239--245.
	\end{thebibliography}
\end{document}